\documentclass[a4paper,11pt]{article}
\usepackage[T1]{fontenc}

\usepackage{tikz}
\usepackage[utf8]{inputenc}
\usepackage{amsmath,amsfonts,amsthm,amssymb, amsopn,mathtools}
\usepackage[final, pdftex, pdfpagelabels, pdfstartview = {FitH}, bookmarks, colorlinks, plainpages = false, linktoc=all, linkcolor=red, citecolor=blue, urlcolor=blue,filecolor=black]{hyperref}
\usepackage{cleveref}
\usepackage{pdfrender, tikz-cd, rotating}
\usepackage{fancyhdr}
\usepackage{amsmath,amsfonts}
\usepackage{amsthm}
\usepackage{amssymb, amsopn}
\usepackage[USenglish]{babel}
\usepackage{enumerate}
\usepackage{verbatim}
\usepackage{graphicx}
\usepackage[left=3cm, right=3cm, top=3cm, bottom=3cm]{geometry}
\usepackage{csquotes}
\usepackage{mathscinet} 
\usepackage{enumitem}

\usepackage[textwidth=3.5cm, color=orange!10]{todonotes}
\usetikzlibrary{matrix,arrows,decorations.pathmorphing,decorations.markings}

\theoremstyle{plain}
\newtheorem{thm}{Theorem}[section]
\newtheorem{lemma}[thm]{Lemma}
\newtheorem{prop}[thm]{Proposition}
\newtheorem{corollary}[thm]{Corollary}

\newtheorem{introthm}{Theorem}
\newtheorem{introprop}[introthm]{Proposition}
\theoremstyle{definition}
\newtheorem{introdef}[introthm]{Definition}

\theoremstyle{definition}
\newtheorem{definition}[thm]{Definition}
\newtheorem{example}[thm]{Example}
\newtheorem{remark}[thm]{Remark}

\theoremstyle{remark}

\renewcommand{\tilde}{\widetilde}
\renewcommand{\bar}{\overline}
\newcommand{\bbC}{\mathbb{C}}
\newcommand{\bbP}{\mathbb{P}}
\newcommand{\bbZ}{\mathbb{Z}}

\newcommand{\bbQ}{\mathbb{Q}}
\newcommand{\bbR}{\mathbb{R}}
\newcommand{\bbG}{\mathbb{G}}

\newcommand{\ra}{\rightarrow}
\newcommand{\xra}{\xrightarrow}
\newcommand{\cu}{\subseteq}
\newcommand{\lra}{\longrightarrow}

\newcommand{\affPhi}{\widehat{\Phi}}
\newcommand{\affW}{\widehat{W}}
\newcommand{\extW}{\affW_{ext}}
\newcommand{\affX}{\widehat{X}}
\newcommand{\affT}{\widehat{T}}
\newcommand{\affDelta}{\widehat{\Delta}}

\newcommand{\sch}[1]{\bar{\Gr_{#1}}}

\newcommand{\calB}{\mathcal{B}}
\newcommand{\calC}{\mathcal{C}}
\newcommand{\calD}{\mathcal{D}}
\newcommand{\calF}{\mathcal{F}}
\newcommand{\calG}{\mathcal{G}}

\newcommand{\calO}{\mathcal{O}}

\newcommand{\frg}{\mathfrak{g}}
\newcommand{\frh}{\mathfrak{h}}
\newcommand{\frk}{\mathfrak{k}}
\newcommand{\frC}{\mathfrak{C}}

\newcommand{\barO}{\bar{\calO}}
\newcommand{\Gr}{\mathcal{G}r}
\newcommand{\Iw}{\mathcal{I}w}
\newcommand{\IC}{\mathrm{IC}}

\newcommand{\bfH}{\mathbf{H}}
\newcommand{\bfN}{\mathbf{N}}
\newcommand{\undH}{\underline{\mathbf{H}}}
\newcommand{\undQ}{\underline{\mathbb{Q}}}
\newcommand{\htil}{\tilde{h}}

\newcommand{\affHec}{\widetilde{\mathcal{H}}}

\newcommand{\aand}{\qquad\text{ and }\qquad}
\newcommand{\Address}{
	\bigskip{\footnotesize
		
		\textsc{Dipartimento di Matematica, Universit\`a di Pisa, Italy}\par\nopagebreak
		\textit{E-mail address}: \texttt{leonardo.patimo@unipi.it}
}}

\DeclareMathOperator{\End}{End}
\DeclareMathOperator{\grdim}{grdim}

\DeclareMathOperator{\ch}{ch}

\DeclareMathOperator{\Ima}{Im}
\DeclareMathOperator{\wt}{wt}
\DeclareMathOperator{\Perv}{Perv}
\DeclareMathOperator{\Sat}{Sat}
\DeclareMathOperator{\Rep}{Rep}
\DeclareMathOperator{\Conv}{Conv}
\DeclareMathOperator{\HL}{HL}
\DeclareMathOperator{\rk}{rk}

\setuptodonotes{size = small}
\title{Charges via the Affine Grassmannian}
\author{Leonardo Patimo}

\begin{document}

\maketitle

\begin{abstract}
We introduce a new approach to constructing charge statistics for $q$-weight multiplicities, which is based on the geometry of the affine Grassmannian of the Langlands dual group. 
We track hyperbolic localization variations under wall crossings in the cocharacter lattice and define a mimicking combinatorial procedure based on iterated swaps to construct the charge.

\end{abstract}

\section*{Introduction}

The dimensions of the weight spaces of irreducible representations of a reductive group $G$ have classical and well-known combinatorial interpretations: for example, for $G=SL_n(\bbC)$, the dimension of the $\mu$-weight space of the irreducible representation of highest weight $\lambda$ is the number of semi-standard Young tableaux $\calB(\lambda)_\mu$ of shape $\lambda$ and content $\mu$.

In 1983 Lusztig \cite{LusSingularities} defined $q$-analogues of 
the weight multiplicities for any reductive group $G$, also known as the Kostka--Foulkes polynomials $K_{\lambda,\mu}(q)$.  These polynomials record the dimensions of the graded components of the Brylinski filtration on weight spaces and can also be obtained as Kazhdan--Lusztig polynomials for the corresponding affine Weyl group. In particular, they have positive coefficients.
Finding a combinatorial interpretation for the coefficients of the Kostka--Foulkes polynomials is one of the most long-standing open problems in combinatorial representation theory.

In type A, Lascoux and Sch\"utzenberger gave a combinatorial formula for the Kostka--Foulkes polynomials by defining a statistic on the set of semi-standard tableaux, called the \emph{charge statistic} \cite{LSSur}. This is a function $c:\calB(\lambda)_\mu\ra \bbZ_{\geq 0}$ such that
\[ K_{\lambda,\mu}(q)=\sum_{T\in \calB(\lambda)_\mu} q^{c(T)}.  \]
There exist alternative constructions of the charge statistic.
In \cite{LLTCrystal}, Lascoux, Leclerc and Thibon gave a formula for the charge statistic in terms of the crystal graphs while Nakayashiki and Yamada \cite{NYKostka} defined the charge using the energy function on affine crystals. The proofs of all these formulas ultimately rely on tableaux combinatorics and are based on an operation called \emph{cyclage} of tableaux.

In this paper, we introduce a different approach to construct charge statistics following different and more geometric methods. Motivated by the geometric Satake equivalence, we relate the charge statistic on $\calB(\lambda)$ to the geometry of the affine Grassmannian $\Gr^\vee$ of the Langlands dual group $G^\vee$ of $G$. To our knowledge, this is the first time that the charge statistic is directly linked to the geometry of the affine Grassmannian. We regard this connection as an important tool to define the charge statistic in other types beyond type $A$.

One of the main ingredients in the geometric Satake equivalence is Braden's hyperbolic localization \cite{BraHyperbolic}. On $\Gr^\vee$ there is an action of the \emph{augmented torus} $\affT=T^\vee\times \bbC^*$, where $T^\vee\cu G^\vee$ is a maximal torus and $\bbC^*$ acts by loop rotations.
Every dominant cocharacter $\eta$ of $\affT$ determines a one-parameter subgroup $\bbC^*\cu \affT$ acting on the affine Grassmannian. For any weight $\mu$, we denote by $\HL^\eta_\mu$ the corresponding hyperbolic localization functor.

If $\eta$ is a dominant $T^\vee$-cocharacter (i.e., it is trivial in the $\bbC^*$-component), then the functors $\HL^\eta_\mu$ are exact on the category of perverse sheaves on $\Gr^\vee$ and they are known as the \emph{weight functors}.
 Let $\IC_\lambda$ denote the intersection cohomology sheaf of a Schubert variety inside $\Gr^\vee$. Then, the weight functors on $\IC_\lambda$ correspond via geometric Satake to the weight spaces of the irreducible representation $L(\lambda)$ of $G$. Remarkably, these weight functors come naturally endowed with a canonical basis parametrized by  \emph{MV cycles}, the irreducible components of the intersection of the $\eta$-attractive sets with the Schubert varieties. The existence of canonical bases parametrized by MV cycles can be interpreted as a geometric explanation for the existence of combinatorial formulas for the dimension of weight spaces for an arbitrary reductive group $G$.

If instead $\eta$ is  a dominant cocharacter of  $\affT$ (i.e., dominant with respect to the affine root system) the $\eta$-attractive sets are affine spaces, namely the Bruhat cells. In this case the functors $\HL^\eta_\mu$ return graded vector spaces whose graded dimensions are, up to a shift, the Kazhdan--Lusztig polynomials.

Thus, we have two regions of the cocharacter lattice $X_\bullet(\affT)$ where $\HL$ gives interesting outputs. If $\eta$ is a dominant $T^\vee$-cocharacter we say that $\eta$ is in the \emph{MV region} (the region where $\HL$ yields the MV cycles).  If $\eta$ is a dominant $\affT$-cocharacter we say that $\eta$ is in the \emph{KL region} (the region where graded dimension of  $\HL$ yields KL polynomials).

Our goal is to ``charge'' MV cycles so that they enumerate the KL polynomials. We achieve this by looking at a family of cocharacters  $\eta(t)$ going from the MV region to the KL region and recording how the hyperbolic localization changes along the way.

We find that the hyperbolic localization only changes when we cross a wall of the form 
\[H_{\alpha^\vee}=\{ \eta \in X_\bullet(\affT) \mid \langle \eta,\alpha^\vee\rangle = 0\}\]
where $\alpha^\vee$ is a positive real root of the affine root system of $G^\vee$.
If $\eta$ is a singular cocharacter lying on a wall $H_{\alpha^\vee}$, each connected component of the $\eta$-fixed points is isomorphic either to a point or to $\bbP^1(\bbC)$.
Understanding how $\HL$ changes between the two sides of the wall can be reduced to a problem on the projective line, which is easy to solve as there are only two simple equivariant perverse sheaves on $\bbP^1(\bbC)$. We denote by $<$ the Bruhat order on weights.

\begin{introprop}
Let $\eta_1$ and $\eta_2$ be two cocharacters which lie on opposite sides of a wall $H_{\alpha^\vee}$ with $\langle \eta_1,\alpha^\vee\rangle <0<\langle \eta_2,\alpha^\vee\rangle$. For any $\affT$-fixed point $\mu$ the graded dimensions of the $\HL$ functors vary as follows:
\[	\begin{array}{lr} \grdim\HL_{\mu}^{\eta_2}(\IC_\lambda) =\; q^{-2}\grdim\HL_{\mu}^{\eta_1}(\IC_\lambda) &\text{ if }s_{\alpha^\vee}(\mu)<\mu,\vspace{2mm}\\
\grdim\HL_{\mu}^{\eta_2}(\IC_\lambda) =\; \grdim\HL_{\mu}^{\eta_1}(\IC_\lambda)+ (1-q^{-2}) \cdot\grdim\HL_{s_{\alpha^\vee}(\mu)}^{\eta_1}(\calF) &\text{ if }s_{\alpha^\vee}(\mu)>\mu.\end{array}\]
\end{introprop}

The main focus of this paper is the definition of a combinatorial procedure that mimics the variation of hyperbolic localization functions along a family of cocharacters.

We can extend the definition of charge to any regular cocharacter. We say that $r(\eta,-)$ is a \emph{renormalized charge} (or \emph{re-charge}) for $\eta$ if
\[\grdim \HL^\eta_\mu(\IC_\lambda)=\sum_{T\in \calB(\lambda)_\mu} q^{2r(\eta,T)}.\]
It is straightforward to construct a recharge for $\eta_{MV}$ in the MV region, where hyperbolic localization is concentrated in a single degree. The matching recharge must be constant on each $\calB(\lambda)_\mu$ and it is defined by the formula
\begin{equation*}
	r(\eta_{MV},T)=-\langle \wt(T),\rho^\vee\rangle.
\end{equation*}

To obtain a recharge in the KL region we can  start with $r(\eta_{MV},-)$ in the MV region and modify it appropriately every time we cross a wall $H_{\alpha^\vee}$ in the cocharacter lattice. The changes that we need to employ can be efficiently encoded into a \emph{swapping function}.

\begin{introdef}
Let $\alpha^\vee$ be a positive real root of  the affine root system of $G^\vee$, let $\eta$ be a cocharacter and $r(\eta,-)$ a recharge for $\eta$. Then a \emph{swapping function} $\psi=\{\psi_\nu\}_{\nu\in X, \nu>s_{\alpha^\vee}(\nu)}$ for $r(\eta,-)$ and $\alpha$ is
  a collection of injective functions $\psi_{\nu}:\calB(\lambda)_{\nu} \ra \calB(\lambda)_{s_{\alpha^\vee}(\nu)}$ such that
\begin{equation}\label{psiintro}
	r(\eta,\psi_\nu(T))=r(\eta,T)-1 \qquad \text{ for every }T\in \calB(\lambda)_\nu. 
\end{equation}	
\end{introdef}

Let $\eta_1$ and $\eta_2$ be two cocharacters lying on opposite sides of a wall $H_{\alpha^\vee}$ with $\langle \eta_1,\alpha^\vee\rangle <0<\langle \eta_2,\alpha^\vee\rangle$. Suppose we have a recharge $r(\eta_1,-)$ for $\eta_1$. Then, there exists a swapping function $\psi$ for $r(\eta_1,-)$ and we can construct a recharge for $\eta_2$ on the other side of the wall $H_{\alpha^\vee}$ as follows.

\begin{introthm}
	Let $r(\eta_2,-)$ be the function on $\calB(\lambda)$ obtained by swapping  the values of $r(\eta_1,-)$ as indicated by $\psi$, i.e. we have
\begin{equation}
	r(\eta_2,T)=\begin{cases}r(\eta_1,T)-1& \text{if }s_{\alpha^\vee}(\wt(T))<\wt(T)\\
		r(\eta_1,T)+1& \text{if }\wt(T)<s_{\alpha^\vee}(\wt(T))\leq \lambda\text{ and }T\in \Ima(\psi)\\
		r(\eta_1,T)& \text{if }\wt(T)\leq s_{\alpha^\vee}(\wt(T))\text{ and }T\not\in \Ima(\psi).
	\end{cases}
\end{equation}
The $r(\eta_2,-)$ is a recharge for $\eta_2$.
\end{introthm}

For any sheaf $\IC_\lambda$ there are only finitely many walls that we need to cross to reach the KL region. Therefore, if we have at our disposal a swapping function $\psi$ as in \eqref{psiintro} for each of these walls, we can eventually produce a charge statistic.

\subsection*{Applications and future work}

Our approach decomposes the problem of finding a charge statistic into smaller subproblems: finding a swapping function for each of the walls that needs to be crossed.
This turns out to be more tractable  in several situations. 

In \cite{ChargeA}, we construct explicitly swapping functions in type $A$ using modified crystal operators, which are operators associated to every positive root obtained by twisting the usual crystal operators by the action of the Weyl group.
This provides a new independent construction for the charge in type $A$, which we also prove to be equivalent to the one originally described by Lascoux and Schützenberger and Lascoux--Leclerc--Thibon.
 
Moreover, in collaboration with J.~Torres \cite{PTAtoms}  we construct explicitly swapping functions in type $C_2$ with similar methods, using a variation of the modified crystal operators. Building on the results presented here, this leads to the construction of a charge in type $C_2$, a case for which a charge statistic was previously unknown.

Aside from type $A$ and $C_2$, the question of finding a combinatorial interpretation of the Kostka--Foulkes polynomials  is still open. As far as we know, only in type $C$ there is a general conjecture  \cite{LecKostka} which has been proven in some special cases \cite{LecCombinatorics,LLCombinatorics,DGTPositive}. Our geometric approach offers a promising  starting point to treat the problem of finding a charge statistic in a uniform way, since the framework developed here works for any reductive group $G$.

It would be interesting to determine an analogue of the swapping function at the level of MV cycles and/or MV polytopes. If this is achieved, then it is possible to produce distinguished bases of the intersection cohomology of the Schubert varieties of the affine Grassmannian. In this way one would obtain a result in a similar vein to what is done in \cite{PatBases} for the finite Grassmannian. 

Recall that $q$-weight multiplicities form a special class of Kazhdan--Lusztig polynomials. Then, in type $A$, it would also be intersesting to investigate the link between the combinatorial interpretations and the new and conjectural combinatorial formulas for Kazhdan--Lusztig polynomials given in \cite{PatCombinatorial} and \cite{BBD+Towards}.

%
%
%
\subsection*{Acknowledgments} 
 We wish to thank Geordie Williamson for pointing out that studying hyperbolic localization for a family of cocharacters could lead to understanding the charge statistic from a geometrical point of view.

\section{Affine Root Systems and Weyl Groups}\label{affineSection}

\subsection{Root Data and Reductive Groups}

We start by fixing a root datum $(\Phi,X,\Phi^\vee,X^\vee)$ where
\begin{itemize}
 \item $X$ is a finitely generated free abelian group, called the \emph{weight lattice}
 \item $\Phi \cu X$ is the finite subset of \emph{roots}
 \item $X^\vee$ denotes the dual of $X$, called the \emph{coweight lattice} 
 \item $\Phi^\vee \cu X^\vee$ is the dual root system, and its elements are called \emph{coroots}.
\end{itemize} 

We denote by $\langle-,-\rangle: X\times X^\vee \ra \bbZ$ the pairing.
We fix a system of positive roots $\Phi_+\cu \Phi$ and let $\Phi^\vee_+\cu \Phi^\vee$ be the corresponding system of positive coroots. 
We assume that our root system $\Phi$ is irreducible and semi-simple, i.e. that $\rk(\bbZ \Phi) = \rk(X)$. 

To the root datum we associate a split reductive algebraic group $G$ over $\bbQ$.\footnote{In the introduction, for simplicity, we considered a reductive group $G$ defined over $\bbC$. Here, and in the rest of the paper, we make a different choice because it is more general and it helps differentiate between the two sides of the Langlands duality.} We denote by $G^\vee$ the Langlands dual group of $G$ over $\bbC$, i.e. the complex algebraic group associated to the root datum $(\Phi^\vee,X^\vee,\Phi,X)$.
Let $T\cu B \cu G$ be a maximal torus and a Borel subgroup of $G$ such that $\Phi_+$ are the non-zero weights in $\mathrm{Lie}(B)$ for the action of $T$.  Similarly, let $T^\vee \cu B^\vee \cu G^\vee$ be a maximal torus and a Borel subgroup of $G^\vee$ such that $\Phi^\vee_+$ are the weights in $\mathrm{Lie}(B^\vee)$ for the action of $T^\vee$.

 The weight lattice $X$ can be naturally identified with the character group $X^\bullet(T)$ of $T$ and with the cocharacter group $X_\bullet(T^\vee)$ of $T^\vee$, so we can write
\[ X = X^\bullet(T)= X_\bullet(T^\vee).\]
Dually, we have $X^\vee=X^\bullet(T^\vee)=X_\bullet(T)$. 

Let $\Delta=\{ \alpha_1,\alpha_2,\ldots,\alpha_n\}\cu \Phi_+$ be the set of simple roots and  let $\Delta^\vee=\{\alpha_1^\vee,\ldots,\alpha_n^\vee\}\cu \Phi_+^\vee$ be the set of simple coroots.
We denote by $\varpi_1,\varpi_2,\ldots,\varpi_n\in X$ the fundamental weights and by $\varpi_1^\vee,\varpi_2^\vee,\ldots, \varpi_n^\vee\in X^\vee$ the fundamental coweights.
We define
\[ \rho:=\frac12 \sum_{\alpha\in \Phi_+}\alpha\aand\rho^\vee:=\frac12 \sum_{\alpha^\vee\in\Phi_+^\vee}\alpha^\vee. \]

Let $\frg^\vee$ be the Lie algebra of $G^\vee$. Let $\frh^\vee$ be the Cartan subalgebra of $\frg^\vee$ and let $\frh$ be its dual.
We have \[\frh\cong X^\vee\otimes_{\bbZ} \bbC\aand\frh^\vee\cong X\otimes_\bbZ \bbC.\]
We can extend $\langle-,-\rangle$ to a pairing $\frh^\vee\times \frh \ra \bbC$.

Let $\theta^\vee$ be the highest coroot in $\Phi^\vee$. We fix a Killing form $(-,-)$ on $\frh$ renormalized so that $(\theta^\vee,\theta^\vee)=2$. This induces an isomorphism $\sigma:\frh^\vee\xra{\sim} \frh$ defined by
\[\langle \lambda,v\rangle = (\sigma(\lambda),v)\quad\text{ for all }\lambda\in \frh^\vee, \; v\in \frh.\]
Via $\sigma$, we also induce a bilinear form on $\frh^\vee$, which we denote again by $(-,-)$. For every root $\alpha\in \Phi$ we denote by $\alpha^\vee$ the corresponding dual coroot in $\Phi^\vee$ which can be obtained as
\[\alpha^\vee=2\frac{\sigma(\alpha)}{(\alpha,\alpha)}.\]

\subsection{The Affine Root System}
We recall from \cite{KacInfinite} the notions of affine Lie algebras and of the related affine root systems.
We can associate to $\frg^\vee$ its \emph{affine Lie algebra}
\[L(\frg^\vee) = \frg^\vee[t,t^{-1}]\oplus \bbC c \oplus \bbC d\]
where $c$ denotes the \emph{canonical central element} and $d$ the \emph{scaling element}. 

We denote by $\affPhi^\vee_{tot}$ the \emph{affine root system} of $L(\frg^\vee)$. This is defined as the set of weights occurring in the adjoint representation of $L(\frg^\vee)$. We recall now how $\affPhi^\vee_{tot}$ can be directly constructed from $\Phi^\vee$.

Let $\frk^\vee=\frh^\vee \oplus \bbC d\cu L(\frg^\vee)$ and let $\frk$ be the dual vector space of $\frk^\vee$. We have $\frk=\frh \oplus \bbC \delta$ where
\[ \langle d,\frh\rangle =\langle \frh^\vee,\delta\rangle = 0\aand\langle d,\delta\rangle =1.\]
 We define the lattices $ \affX^\vee=X^\vee\oplus \bbZ \delta \subset \frk$ and $\affX = X \oplus \bbZ d\subset \frk^\vee$. We extend the Killing form from $\frh$ to a bilinear form on $\frk$ (always denoted by $(-,-)$) by setting
 \[ (\delta,v) =0\quad \text{ for any }v\in \frk.\]

The affine root system $\affPhi^\vee_{tot}$ can be realized inside $\affX^\vee$. There are two types of roots in $\affX^\vee$: real roots  
and imaginary roots. The imaginary roots and the real roots are respectively \[\affPhi^\vee_{im}=\{ k\delta \mid k \in \bbZ \setminus \{0\}\}\aand \affPhi_{re}^\vee= \{ \alpha^\vee + m\delta \mid \alpha^\vee \in \Phi^\vee, m \in \bbZ\}.\]
The positive real roots are \[\affPhi^\vee_{re,+} =\{ \alpha^\vee + m\delta \mid \alpha^\vee \in \Phi^\vee, m >0\} \cup \Phi^\vee_+\]
 and the simple roots are 
 \[\affDelta^\vee = \Delta^\vee \cup \{\alpha^\vee_0\}\]
where $\alpha^\vee_0:=\delta -\theta^\vee$. 
 Only real roots play a role in this paper, so we simply use the notation $\affPhi^\vee$ for $\affPhi^\vee_{re}$ and $\affPhi^\vee_+$ for $\affPhi^\vee_{re,+}$.

For any simple roots $\alpha_i^\vee\in \affDelta^\vee$ we define the corresponding simple reflection $s_i:\frk\ra \frk$ by
\[s_i(v) = v-2\frac{(v,\alpha_i^\vee) }{(\alpha_i^\vee,\alpha_i^\vee)}\alpha_i^\vee.\]

\begin{definition}
 We denote by $\affW$ the subgroup of $\End_{\bbC}(\frk)$ generated by $s_0,\ldots,s_n$. The group $\affW$ is called the \emph{affine Weyl group} of $\Phi^\vee$.
\end{definition}

The subgroup generated by $s_1,\ldots,s_n$ is isomorphic to the Weyl group $W$ of $G^\vee$ and of $G$. 
We have an isomorphism
\[\affW\cong W \ltimes \bbZ\Phi.\]
Each real root can be conjugated to a simple root via the affine Weyl group. The reflections in $\affW$ are in bijection with
$\affPhi^\vee_+$. For $\alpha^\vee\in \affPhi^\vee_+$ we denote by $s_{\alpha^\vee}$ the corresponding reflection, while for a reflection $t\in \affW$ we denote by $\alpha_t^\vee$ the corresponding positive real root.
\begin{remark}
It may appear at first counterintuitive to see that it is the root lattice $\bbZ\Phi$ (and not $\bbZ\Phi^\vee$!) that occurs as a subgroup of linear endomorphisms of $\frk$. To clarify this, we explain how one can identify $\lambda \in \bbZ \Phi$ with an element $\mathtt{t}_\lambda\in \affW\cu \End_{\bbC}(\frk)$.

Let $\theta\in \Phi$ be the dual root of $\theta^\vee$ (i.e.,  the highest short root in $\Phi_+$). Since $\theta$ is a short root, the following holds.
\begin{equation}\label{Wtheta}
\text{The orbit }W\cdot \theta\text{ generates the lattice }\bbZ \Phi. 
\end{equation}

For any $w\in W$, we have $w\mathtt{t}_{\lambda} w^{-1}=\mathtt{t}_{w(\lambda)}$. Because of \eqref{Wtheta}, it is enough to describe the element $\mathtt{t}_\theta\in \affW$. This is given by the formula
\begin{equation}\label{ttheta}\mathtt{t}_\theta=s_{\theta^\vee}s_0.\footnote{Observe that this is the opposite choice of \cite{KacInfinite}. The reason for this difference is that we want the simple reflections $s_0,\ldots,s_n$ to be the reflections along the walls of the fundamental alcove \[C^-=\{\lambda \in X_{\bbR} \mid -1 \leq \langle \lambda,\alpha^\vee\rangle \leq 0\text{ for all }\alpha\in \Phi_+ \}.\]}\end{equation}
A more explicit description of the action of $\bbZ \Phi$ on $\frk$ can be found in \cite[(6.5.2)]{KacInfinite}.
\end{remark}

\subsection{The Affine Action of \texorpdfstring{$\affW$}{the Affine Weyl Group}}\label{AffineAction}

The action of $\affW$ on $\frk$ described above occurs naturally as the action of the Weyl group of $L(\frg^\vee)$ on the dual of the Cartan subalgebra. 
There is another important action of $\affW$ that we need to consider at the same time: the action on $\frh^\vee$ by affine transformations.

Recall that we have $\affW=W\ltimes \bbZ\Phi$ and $\bbZ \Phi$ is a subgroup of $X\cu \frh^\vee$. 
The finite Weyl group $W$ acts on $\frh$, and acts as well on $\frh^\vee$ via the isomorphism $\sigma$. The root lattice $\bbZ \Phi$ acts on
$\frh^\vee$ by translations. We can combine these two actions to obtain a faithful action of $\affW$ on $\frh^\vee$, i.e., we realize $\affW$ as a subgroup of $\mathrm{Aff}_\bbC(\frh^\vee)$, the group of affine transformations of $\frh^\vee$. We will refer to this as the \emph{affine action} of $\affW$. Notice that the affine action restricts to an action of $\affW$ on $X$.

We need to make more explicit how the reflections in $\affW$ act on $\frh^\vee$. Recall by \eqref{ttheta} that for the simple reflection $s_0$ (which corresponds to the simple root $\alpha_0^\vee=\delta-\theta^\vee$) we have $s_0= s_{\theta^\vee} \mathtt{t}_\theta$, so $s_0$ acts on $\frh^\vee$ as the affine reflection
\[ s_0(\lambda) = \lambda - \left(2\frac{(\lambda, \theta)}{(\theta,\theta)} +1 \right)\theta = \lambda - \left(\langle \lambda,\theta^\vee\rangle +1 \right)\theta.\]
More generally, if $\alpha^\vee=m\delta- \beta^\vee\in \affPhi^\vee$ is a real root, with $m\in \bbZ$ and $\beta^\vee \in \Phi^\vee$, we have 
\begin{equation}\label{salpha0}
s_{\alpha^\vee}(\lambda) = \lambda - \left(2\frac{ (\lambda,\beta)}{(\beta,\beta)} + m	\right)\beta = \lambda - \left(\langle \lambda,\beta^\vee\rangle + m	\right)\beta.
\end{equation}

The affine action is not free. However, we have the following weaker statement regarding reflections. 
\begin{lemma}\label{notfixing}
	Let $r,t\in\affW$ be reflections and let $\mu \in \frh^\vee$. If $t\mu=r\mu \neq \mu$, then $r=t$.
\end{lemma}
\begin{proof}
This is actually a general statement about affine reflections. If $r\mu=t\mu$, then the line between $\mu$ and $r\mu$ is orthogonal to both the hyperplanes of $r$- and $t$-fixed points. So these hyperplanes are parallel and the only possibility to have $r\mu=t\mu$ is $r=t$.
\end{proof}

\subsection{The Extended Affine Weyl Group}\label{extendedSection}

\begin{definition}
The \emph{extended affine Weyl group} $\extW$ is the subgroup of affine transformations of $\frh^\vee$ generated by $W$ and by the translations $\mathtt{t}_\lambda$, for $\lambda\in X$. 
\end{definition}

We have
\begin{equation}\label{Wextdec}
\extW \cong W\ltimes X.
\end{equation}
The group $\extW$ is usually not a Coxeter group, but we can define a length function $\ell$ on $\extW$, in terms of the isomorphism \eqref{Wextdec} as follows:
\[\ell(\mathtt{t}_\lambda w)=\sum_{\substack{\alpha\in \Phi_+\\ w(\alpha)\in \Phi_+}} |\langle \lambda,\alpha^\vee\rangle| +\sum_{\substack{\alpha\in \Phi_+\\ w(\alpha)\in -\Phi_+}} |\langle \lambda,\alpha^\vee\rangle+1| \]
for $w\in W$ and $\lambda \in X$. This extends the usual length function of $\affW$. 

The group $\extW$ contains $\affW$ as a normal subgroup. 
Let $\Omega \cu \extW$ be the subgroup of elements of length $0$. We have
\[ \Omega = \extW/\affW \cong X /\bbZ \Phi\]
and the group $\Omega$ acts by conjugation on $\affW$ via automorphisms of the Dynkin diagram. We have
\begin{equation}\label{Omegadec}
 \extW \cong \Omega \ltimes \affW.
\end{equation}

In view of \eqref{Omegadec}, we also define a Bruhat order on $\extW$ by setting $\omega_1 w_1\leq \omega_2 w_2$ for $\omega_1,\omega_2\in \Omega$ and $w_1,w_2\in \affW$ if $\omega_1=\omega_2$ and $w_1\leq w_2$ in the Bruhat order of $\affW$. Moreover, for any $w\in \extW$ we have $\ell(w)=|\{\alpha^\vee\in \affPhi^\vee_+ \mid \ell(s_{\alpha^\vee}w)<\ell(w)\}|$.

We have
\[\extW/W \cong X \aand W\backslash \extW / W\cong X_+\]
where $X_+=\{\lambda\in X \mid \langle \lambda , \beta^\vee \rangle \geq 0$ for all $\beta^\vee \in \Phi^\vee_+\}$.

We can then regard $X$ as the quotient $\extW/W$. This induces a length function on $X$, where $\ell(\lambda)$ is defined to be the minimal length of an element in its coset. This also induces a Bruhat order on $X$, which we denote again by $\leq$. The length function on $X$ can be expressed as follows.
\begin{equation}\label{totallength} \ell(\lambda)=\sum_{\substack{\alpha\in \Phi_+\\ \langle \lambda,\alpha^\vee\rangle\geq 0}} \langle \lambda,\alpha^\vee\rangle -\sum_{\substack{\alpha\in \Phi_+\\ \langle \lambda,\alpha^\vee\rangle< 0}} \left(\langle \lambda,\alpha^\vee\rangle+1\right). 
\end{equation}
Notice that if $\lambda\in X_+$ then $\ell(\lambda)=2\langle \lambda,\rho^\vee\rangle$. 

Let $\alpha^\vee=m\delta+\beta^\vee$, with $\beta^\vee\in \Phi_+^\vee$. For $\lambda\in X$, we have 
	\begin{equation}\label{bruhatorder}s_{\alpha^\vee}(\lambda)< \lambda \iff \begin{cases}\langle \lambda,\beta^\vee\rangle > m &\text{if }m\geq 0,\\
\langle \lambda,\beta^\vee\rangle <  m&\text{if }m<0.\end{cases}\end{equation}
The Bruhat order on $X$ is generated by the relations $s_{\alpha^\vee}(\lambda)<\lambda$ as above, for $\lambda\in X$ and $\alpha^\vee\in \affPhi^\vee$.

If $\lambda\in X_+$, then 
\[\mu\leq \lambda\iff \mu\in \Conv_\bbR(W\cdot \lambda) \cap (\lambda +\bbZ \Phi),\]
where $\Conv_\bbR(W\cdot \lambda)$ denotes the convex hull in $X_\bbR:=X\otimes_\bbZ \bbR$ of the $W$-orbit of $\lambda$. In particular, on $X_+$ the Bruhat order coincides with the dominance order on weights, i.e., if $\mu\in X_+$, then $\mu\leq \lambda$ if and only if $\lambda-\mu$ is a positive integral combination of simple roots.

\section{Hyperbolic Localization on the Affine Grassmannian}\label{AffineSEC}
\subsection{The Loop Group and the Affine Grassmannian}
 We now consider the affine Grassmannian of the Langlands dual group $G^\vee$.\footnote{We need to work on both sides of the geometric Satake equivalence, and this forces the notation to be rather cumbersome on one of the two sides.
We consider the representations of $G$ to be the main object of the present paper. For this reason, the geometric side is regarded as the ``dual'' side and the notation adopted ($G^\vee,\Gr^\vee, \affPhi^\vee,\ldots$) reflects this.}
 \[\Gr^\vee=G^\vee((t))/G^\vee[[t]].\]

We have an action of the \emph{augmented torus} $\affT:=T^\vee\times \bbC^*$ on $\Gr^\vee$, where $T^\vee\cu G^\vee$ is the maximal torus and $\bbC^*$ acts by loop rotation (i.e. $z\in \bbC^*$ acts by sending indeterminate $t$ to $z t$).

The character group $X^\bullet(\affT)$ can be identified with $\affX^\vee=X^\vee \oplus \bbZ \delta$ where we  interpret $\delta$ as the character
\[ \delta:\affT=T^\vee\times \bbC^*\ra \bbC^*,\qquad \delta(t,z)=z.\]
Similarly, the cocharacter group $X_\bullet(\affT)$ can be identified with $\affX=X\oplus \bbZ d$, where
we interpret $d$ as the cocharacter 
\[d:\bbC^* \ra \affT,\qquad d(z)=(1,z).\]
The Lie algebra of $\affT$ is naturally identified with $\frk^\vee$.
Every weight $\mu\in X$ induces a cocharacter $\mu:\bbC^*\ra T^\vee$. We reinterpret it as a point of $T^\vee[t,t^{-1}]\cu G^\vee((t))$ which we denote by $\tilde{L}_\mu$. We denote by $L_\mu$ the projection of $\tilde{L}_\mu$ to $\Gr^\vee$.

The set of $\affT$-fixed points on $\Gr^\vee$ coincides with the set of $T^\vee$-fixed points and they are exactly the points, for $\mu\in X$, so we have a bijection
\begin{align}\label{fixedpointsbij}
 X \quad\cong &\quad \left\{ \begin{array}{c}\affT\text{-fixed points on }\Gr^\vee\end{array}\right\} \\
 \mu \quad\mapsto &\qquad\qquad L_\mu.\nonumber
\end{align} 

\subsubsection{Root Subgroups of the Loop Group}\label{secRootSG}
For any $\beta^\vee \in \Phi^\vee$ there exists a one dimensional subgroup $U_{\beta^\vee}\subset G^\vee$, called the \emph{root subgroup}, which is isomorphic to the additive group 
\[ u_{\beta^\vee}:\bbG_a \xra{\sim} U_{\beta^\vee}\]
and such that for any $t\in T^\vee$ we have $t u_{\beta^\vee}(x)t^{-1} = u_{\beta^\vee} (\beta^\vee(t)x)$.
For any $\beta^\vee\in \Phi^\vee_+$, the groups $U_{\beta^\vee}$ and $U_{-\beta^\vee}$ generate a subgroup of $G^\vee$ isomorphic to $SL_2$ or $PGL_2$, so we have a surjective map
\[ \psi_{\beta^\vee}:SL_2(\bbC) \ra \langle U_{-\beta^\vee},U_{\beta^\vee}\rangle\]
such that 
\[\psi_{\beta^\vee}\left(\left\{\begin{pmatrix}
1 & x\\ 0 & 1
\end{pmatrix}\right\}_{x\in \bbC}\right)=U_{\beta^\vee}\aand \psi_{\beta^\vee}\left(\left\{\begin{pmatrix}
1 & 0\\ x & 1
\end{pmatrix}\right\}_{x\in \bbC}\right)=U_{-\beta^\vee}.\]

\begin{definition}\label{affinerootsubgroups}
 For $\alpha^\vee=m\delta +\beta^\vee\in \affPhi^\vee$ (resp. $\alpha^\vee=m\delta -\beta^\vee\in \affPhi^\vee$) we define the root subgroup $U_{\alpha^\vee}$ of $G^\vee((t))$ to be the image of the additive subgroup
\[\left\{\begin{pmatrix}
1 & t^m x\\ 0 & 1
\end{pmatrix}\right\}_{x\in \bbC} \qquad \left(\text{resp.}\;\left\{\begin{pmatrix}
1 & 0 \\ t^m x & 1
\end{pmatrix}\right\}_{x\in \bbC}\right) \]
under the morphism $\psi_{\beta^\vee}:SL_2(\bbC((t)))\ra G^\vee((t))$. We denote by $u_{\alpha^\vee}:\bbG_a\ra U_{\alpha^\vee}$ the isomorphism given by $u_{\alpha^\vee}(x)=\psi_{\beta^\vee}\begin{pmatrix}
1 & t^m x\\ 0 & 1
\end{pmatrix}$ (resp. $u_{\alpha^\vee}(x)=\psi_{\beta^\vee}\begin{pmatrix}
1 & 0\\ t^m x & 1
\end{pmatrix}$)
\end{definition}

\subsubsection{Bruhat Decomposition}
 Let $\pi: G^\vee[[t]]\ra G^\vee$ be the map that sends $t$ to $0$. Let $\Iw=\pi^{-1}(B^\vee)$ be the \emph{Iwahori subgroup}. The \emph{Bruhat decomposition} for the affine Grassmannian takes the following form
\[\Gr^\vee = \bigsqcup_{\mu \in X} \Iw \cdot L_\mu \]
and an orbit $\calC_\mu:=\Iw \cdot L_\mu$ is called \emph{Bruhat cell}. Each Bruhat cell $\calC_\mu$ is isomorphic to an affine space of dimension $\ell(\mu)$. In fact, we have an isomorphism
\begin{equation}\label{cell}
 \prod_{\substack{\alpha^\vee\in \affPhi^\vee_+\\ s_{\alpha^\vee}(\mu)<\mu}} U_{\alpha^\vee} \xra{\sim} \calC_\mu
\end{equation} 
defined by $u \mapsto u \cdot L_\mu$.
We can also decompose $\Gr^\vee$ into $G^\vee[[t]]$-orbits. We have
\[\Gr^\vee = \bigsqcup_{\mu \in X_+} G^\vee[[t]] \cdot L_\mu.\]
We denote by $\Gr_\lambda$ the $G^\vee[[t]]$-orbit of $L_\lambda$ in $\Gr^\vee$ and we call the closure $\bar{\Gr_\lambda}$ a \emph{Schubert variety}. The Schubert variety $\bar{\Gr_\lambda}$ is an irreducible normal projective variety of dimension $2\langle \lambda,\rho^\vee\rangle$, for $\lambda\in X_+$.

The affine Weyl group $\affW$ can also be realized as the Weyl group of the loop group $G^\vee((t))$, i.e., if $N^\vee\cu G^\vee((t))$ is the subgroup generated by $\{\tilde{L}_{\mu}\}_{\mu \in X}$ and $N_{G^\vee}(T^\vee)$, then
$\affW \cong N^\vee/T^\vee$
(see also \cite[Chapter 13]{KumKac}).
Hence, there is a natural action of $\affW$ on the set of $T^\vee$-fixed points of $\Gr^\vee$ which coincides with the affine action of $\affW$ on $X$ described in \Cref{AffineAction}, i.e. for $w\in \affW$ and $\mu\in X$ we have $w\cdot \tilde{L}_\mu=\tilde{L}_{w\cdot \mu}$. 

\subsection{The Geometric Satake Equivalence}

The main link between the representation theory of $G$ and the geometry of the affine Grassmannian of $G^\vee$ is provided by the geometric Satake equivalence.
Let $\calD^b_{G^\vee[[t]]}(\Gr^\vee)$ denote the equivariant derived category of constructible sheaves of $\bbQ$-vector spaces on $\Gr^\vee$ (cf. \cite{BLEquivariant} and \cite[\S A.3]{BRNotes}). We denote by $\Perv_{G^\vee[[t]]}(\Gr^\vee)$ the full subcategory of equivariant perverse sheaves. 
There exists a convolution product $*$ which makes $\Perv_{G^\vee[[t]]}(\Gr^\vee)$ a symmetric monoidal category.

 Let $\Rep_\bbQ G$ denote the category of finite dimensional representations of $G$ over $\bbQ$.

\begin{thm}[Geometric Satake equivalence \cite{MVGeometric}]
There is an equivalence of monoidal categories
\[\Sat: (\Perv_{G^\vee[[t]]}(\Gr^\vee),*)\xra{\sim} (\Rep_\bbQ G, \otimes_\bbQ)\]
called the \emph{Satake equivalence}.
\end{thm}

For $\lambda\in X_+$ we denote by $\IC_\lambda:=\IC(\sch{\lambda},\bbQ)$ the intersection cohomology sheaf of $\sch{\lambda}$. Then $\Sat(\IC_\lambda)$ is an irreducible representation of $G$ of highest weight $\lambda$.

Let $\affHec$ denote the spherical Hecke algebra attached to our root datum $(\Phi,X,\Phi^\vee,X^\vee)$ \cite{KnoKazhdan}. We adopt the notation of \cite[\S 2.2]{LPPPre}. The Hecke algebra $\affHec$ has a standard basis $\{\bfH_{\lambda}\}_{\lambda\in X_+}$ and a Kazhdan--Lusztig basis $\{\undH_{\lambda}\}_{\lambda \in X_+}$. These are bases of $\affHec$ over $\bbZ[v,v^{-1}]$. We can write 
\begin{equation}\label{KLbasis}\undH_\lambda = \sum_{\mu \leq \lambda} h_{\mu,\lambda}(v)\bfH_\mu
\end{equation}
where $h_{\lambda,\lambda}(v)=1$ and $h_{\mu,\lambda}(v)\in v\bbZ_{\geq 0}[v]$ for any $\mu\leq \lambda$. The polynomials $h_{\mu,\lambda}(v)$ are called \emph{Kazhdan--Lusztig polynomials}.

For any object $\calF \in \calD^b_{G^\vee[[t]]}(\Gr^\vee)$ we define its \emph{character} $\ch(\calF)$ as follows.
\[\ch(\calF) = \sum_{\substack{\mu\leq \lambda\\\mu 
\in X_+}} \grdim (i_{\mu}^*\calF)
v^{-\ell(\lambda)}\bfH_{\mu}\] 
where $\grdim$ denotes the \emph{graded dimension}, that is \[\grdim(\calG)=\sum \dim H^{-i}(\calG)v^i\quad\text{ for any }\quad\calG\in \calD^b(pt).\]

By \cite{KLSchubert}, the character of the simple perverse sheaves form the KL basis, i.e. for any $\lambda \in X_+$, we have 
\begin{equation}
\label{chIC}\ch(\IC_\lambda)=\undH_\lambda.
\end{equation}

It follows that the split Grothendieck group of $\Perv_{G^\vee[[t]]}(\Gr^\vee)$ is isomorphic to the subalgebra of $\affHec$ generated by $\{\undH_{\lambda}\}_{\lambda\in X_+}$ over $\bbZ$.

\begin{remark}
In the literature KL polynomials for the affine Grassmannian 
are also known as the \emph{Kostka--Foulkes polynomials} $K_{\lambda,\mu}(q)$, although a slightly different normalization is used (see \cite[Theorem 1.8]{KatSpherical}).
For every $\mu,\lambda\in X_+$ we have
	\[K_{\lambda,\mu}(q)=h_{\mu,\lambda}(q^{\frac12}).\]

%

There is a third basis of $\affHec$ which interpolates between the standard and the KL bases.
\begin{definition}
	For any $\lambda\in X_+$ let
	\[ \bfN_{\lambda}= \sum_{\mu\leq \lambda}v^{2\langle \lambda-\mu,\rho^\vee\rangle} \bfH_\mu.\]
\end{definition}
Equivalently, $\bfN_{\lambda}$ is the character of the constant sheaf $\undQ_{\sch{\lambda}}[\dim \sch{\lambda}]$. 

\begin{thm}[$q$-atomic decomposition of characters \cite{LasCyclic,ShiMulti}]\label{atomicchar}
	In type $A$, for any $\lambda \in X_+$ we have
	\[ \undH_{\lambda}= \sum_{\mu\leq \lambda} a_{\mu,\lambda}(v) \bfN_{\mu}\]
	where $a_{\mu,\lambda}\in \bbZ_{\geq 0}[v]$.
\end{thm}
\end{remark}

The coefficients $a_{\mu,\lambda}(v)$ have been explicitly  computed in type $A_2$ in \cite{LPAffine}, in type $A_3$ and $A_4$ in \cite{LPPPre}.

\subsection{The Moment Graph of the Affine Grassmannian}\label{MGsection}

Recall from \cite{BMMoment} the definition of the moment graph. This is a labeled directed graph attached to a normal variety with an action of a torus $T$ such that the vertices are the $T$-fixed points and the edges are the one-dimensional $T$-orbits.
The aim of this section is to describe the moment graph of the Schubert varieties $\sch{\lambda}$ with respect to the $\affT$-action.

We know from \eqref{fixedpointsbij} that the $\affT$-fixed points in $\Gr^\vee$ are in bijection with $X$. The Bruhat decomposition induces a Whitney stratification of $\sch{\lambda}$ and $L_\mu\in \sch{\lambda}$ if and only if $\mu\leq \lambda$ in the Bruhat order.

In order to compute the one-dimensional orbits of $\affT$ we can reduce ourselves to consider a single Bruhat cell. By \eqref{cell}, each Bruhat cell $\calC_\lambda$ is isomorphic as a $\affT$-variety to the linear representation 
\begin{equation}\label{weightdec} \bigoplus_{\substack{\alpha^\vee \in \affPhi^\vee_+\\s_{\alpha^\vee}(\lambda)<\lambda}}
V_{\alpha^\vee}\end{equation}
where we denote by $V_{\nu}$ the one-dimensional $\affT$-module of weight $\nu\in \affX^\vee$.

\begin{lemma}\label{onedimorbits}
	There exists a one-dimensional $\affT$-orbit between $L_\mu$ and $L_\lambda$ 
if and only if there exists $\alpha^\vee\in \affPhi_+^\vee$ such that $s_{\alpha^\vee}(\lambda)=\mu$.
\end{lemma}
\begin{proof}
	Let $\calO$ be a one-dimensional $\affT$-orbit. Then $\calO\subset \calC_\nu$ for some $\nu\in X$.	
	As a $\affT$-variety we have $\calC_\nu\cong \bigoplus V_{\alpha^\vee}$ as in \eqref{weightdec}.
	Let $v\in \calO$ and write $v=\sum_{\alpha^\vee\in I} v_{\alpha^\vee}$ with $v_{\alpha^\vee}\in V_{\alpha^\vee}\setminus \{0\}$. Then $\mathrm{Stab}(\calO)\subset \bigcap_{\alpha^\vee\in I} \mathrm{Stab}(\alpha^\vee)$ and since $\mathrm{Stab}(\calO)$ is a subtorus of codimension one, there can be only one element in $I$, that is we have $v=v_{\alpha^\vee}$ and $\calO= V_{\alpha^\vee}=U_{\alpha^\vee}\cdot L_\nu$. From this we see that $\nu\in \barO$ and the only other point in $\barO\setminus \calO$ is   $\lim_{z\ra \infty} u_{\alpha^\vee}(z)\cdot L_\nu$,
	 where $u_{\alpha^\vee}:\bbG_a\xra{\sim} U_{\alpha^\vee}$ is the isomorphism in \Cref{affinerootsubgroups}. This proves one direction since $\lim_{z\ra \infty} u_{\alpha^\vee}(z)\cdot L_\nu=L_{s_{\alpha^\vee}(\nu)}$ as it follows from \cite[Theorem 13.2.8 and Exercise 13.2(5)]{KumKac} and the corresponding statement for Kac--Moody groups.
	 
	 In the other direction, if $\mu=s_{\alpha^\vee}(\lambda)$ and $\mu<\lambda$, then the summand  $V_{\alpha^\vee}$ in \eqref{weightdec} is a one-dimensional orbit between $\mu$ and $\lambda$.
\end{proof}

 Moreover, if $\calO$ is a one-dimensional orbit between $\lambda$ and $s_{\alpha^\vee}(\lambda)$ the label of the corresponding edge is $\alpha^\vee\in \frk=\mathrm{Lie}(\affT)^\vee$ (see also \cite[\S 2.2]{BMMoment}). In the following simple Lemma we determine for which pairs of weights $\mu,\lambda$ there exists $\alpha^\vee$ such that $s_{\alpha^\vee}(\lambda)=\mu$. 

\begin{lemma}\label{lemmaedges}
	There exists a reflection $s_{\alpha^\vee}$ such that $s_{\alpha^\vee}(\lambda)=\mu$ if and only if $\lambda-\mu$ is a multiple of a root $\beta\in \Phi$.
	
	Moreover, if $\lambda-\mu$ is a multiple of $\beta$, then $\alpha^\vee=m\delta-\beta^\vee$ where 
	\[ m = -\frac{\langle \lambda+\mu,\beta^\vee\rangle }{2}.\]
\end{lemma}
\begin{proof}
	Assume that $\alpha^\vee=m\delta-\beta^\vee$ and that $s_{\alpha^\vee}(\lambda)=\mu$. We see from \eqref{salpha0} that $\mu=s_{\alpha^\vee}(\lambda)$ only if $\lambda-\mu$ is a multiple of the root $\beta$. Moreover, we have 
	\[\lambda -\mu - \langle \lambda ,\beta^\vee\rangle\beta= m \beta\]
	 and applying $\beta^\vee$ to both sides we obtain the desired claim.
\end{proof}

\begin{remark}\label{remarkacttranslation}
	As in \Cref{extendedSection}, we can also think of a weight $\lambda\in X$ as the coset in $\extW/W$ containing the translation $\mathtt{t}_\lambda$.
	This identification is compatible with the affine action of $\extW$ on $X$.
	In fact, for $\lambda\in X$ an $\alpha^\vee=m\delta -\beta^\vee \in \affPhi^\vee$  with $m\in \bbZ$ and $\beta^\vee\in \Phi^\vee$, by \eqref{salpha0} we have
	\[ s_{\alpha^\vee} \mathtt{t}_\lambda(\nu)=
	\nu+\lambda-(\langle \lambda+\nu,\beta^\vee\rangle +m)\beta=\nu-\langle \nu,\beta^\vee\rangle \beta + \lambda - (\langle \lambda,\beta^\vee\rangle+m)\beta= \mathtt{t}_{s_{\alpha^\vee}(\lambda)}s_{\beta^\vee}(\nu).\]
	Since $s_\beta^\vee\in W$, it follows that $s_{\alpha^\vee}\mathtt{t}_\lambda W = \mathtt{t}_{s_{\alpha^\vee}(\lambda)}W$. In particular, since the reflections $s_{\alpha^\vee}$ generate $\affW$, we have $w\mathtt{t}_\lambda W=\mathtt{t}_{w(\lambda)} W$ for any $w\in \affW$. \end{remark}

We can now describe the moment graph $\Gamma_\lambda$ of $\bar{\Gr_\lambda}$, for $\lambda \in X_+$. The vertices of the graph are all the weights $\mu\in X$ such that $\mu \leq \lambda$ in the Bruhat order. We have an edge $\mu_1\ra \mu_2$ in $\Gamma_\lambda$ if and only if $\mu_2-\mu_1$ is a multiple of a root $\beta\in \Phi$ and $\mu_1\leq \mu_2$. In this case, the label of the edge $\mu_1\ra \mu_2$ is $m\delta-\beta^\vee$, where $m$ can be obtained  as in \Cref{lemmaedges}.

We also define the moment graph $\Gamma_X$ of $\Gr^\vee$ as the union of all the graphs $\Gamma_\lambda$, for all $\lambda\in X$. We call $\Gamma_X$ the \emph{Bruhat graph} of $X$.

\begin{example}\label{A2}
	Suppose that $\Phi$ is a root system of type $A_2$ with simple roots $\alpha_1$ and $\alpha_2$. Let $\lambda=\alpha_{12}:=\alpha_1+\alpha_2$.
	The following is the moment graph $\Gamma_\lambda$ (equal colors correspond to equal labels). 
	
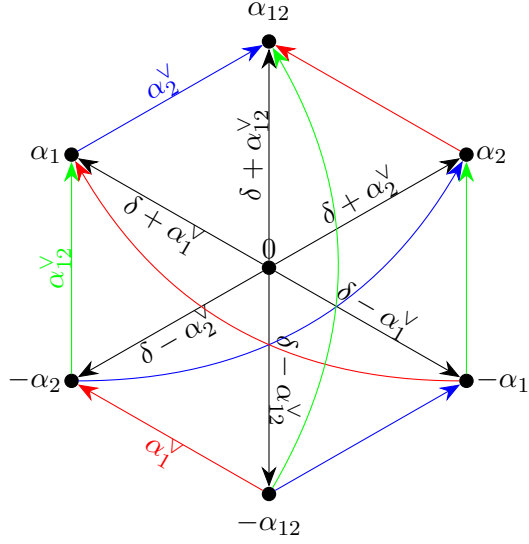
\begin{center}
		\begin{tikzpicture}
		\path (0,3.4) node {$\alpha_{12}$};
		\path (0,-3.4) node {$-\alpha_{12}$};	
		\tikzstyle{every node}=[draw,circle,fill=black,minimum size=5pt,
		inner sep=0pt]	
		\draw (0,0) node (0) [label=above:$0$] {};
		\draw (30:3) node (2)[label=right:$\alpha_2$] {};
		\draw (90:3) node (12) {};
		\draw (150:3) node (1)[label=left:$\alpha_1$] {};
		\draw (-30:3) node (-1)[label=right:$-\alpha_1$] {};
		\draw (-90:3) node (-12) {};
		\draw (-150:3) node (-2)[label=left:$-\alpha_2$] {};
		\tikzstyle{every node}=[inner sep=0pt]
		\draw[-{Stealth[scale=1.6]}] (0) -- node[below,sloped] {$\delta+\alpha_1^\vee$} (1);
		\draw[-{Stealth[scale=1.6]}] (0) -- node[above,sloped] {$\delta+\alpha_2^\vee$} (2);
		\draw[-{Stealth[scale=1.6]}] (0) -- node[above,sloped] {$\delta+\alpha_{12}^\vee$} (12);
		\draw[-{Stealth[scale=1.6]}] (0) -- node[above,sloped] {$\delta-\alpha_1^\vee$} (-1);
		\draw[-{Stealth[scale=1.6]}] (0) -- node[below,sloped] {$\delta-\alpha_2^\vee$} (-2);
		\draw[-{Stealth[scale=1.6]}] (0) -- node[above,sloped] {$\delta-\alpha_{12}^\vee$} (-12);
		\draw[-{Stealth[scale=1.6]},blue] (1) -- node[above,sloped] {$\alpha_2^\vee$} (12);
		\draw[-{Stealth[scale=1.6]},green] (-2) -- node[above,sloped] {$\alpha_{12}^\vee$} (1);
		\draw[-{Stealth[scale=1.6]},red] (-12) -- node[below,sloped] {$\alpha_1^\vee$} (-2);
		\path[-{Stealth[scale=1.6]},blue,bend right] (-2) edge (2);
		\path[-{Stealth[scale=1.6]},red,bend left] (-1) edge (1);
		\path[-{Stealth[scale=1.6]},green,bend right] (-12) edge (12);
		\path[-{Stealth[scale=1.6]},blue] (-12) edge (-1);
		\path[-{Stealth[scale=1.6]},green] (-1) edge (2);
		\path[-{Stealth[scale=1.6]},red] (2) edge (12);
		\end{tikzpicture}
	\end{center}
\end{example}

\begin{remark}
A description of the moment graph of $\Gr^\vee$ can also be obtained from the results of Atiyah and Pressley in \cite{APConvexity}. They computed the moment map of the affine Grassmannian and showed that the image is contained in the paraboloid 
$\{ m\delta+v \mid  v\in \frh\text{ and }m=(v,v)\}\cu \frk$. 	
\end{remark}

\subsection{Hyperbolic Localization on the Affine Grassmannian}

We first recall hyperbolic localization in a general framework. 
Suppose we have a normal complex variety $Y$ with an action of a rank $1$ torus $\bbC^*$. Let $F\subset Y^{\bbC^*}$ be a connected component of the set of $\bbC^*$-fixed points. We consider the \emph{attractive set} 
\[Y^+:=\{ y \in Y \mid \lim_{t\ra 0} t\cdot y\in F\}\]
and the \emph{repulsive set}
\[Y^-:=\{ y \in Y \mid \lim_{t\ra \infty} t\cdot y\in F\}\]
We have the following commutative diagram of inclusions.
\begin{center}
	\begin{tikzpicture}
	\node (a) at (0,2) {$F$};
	\node (b) at (2,2) {$Y^+$};
	\node (c) at (0,0) {$Y^-$};
	\node (d) at (2,0) {$Y$};
	\path[right hook->] 
	(a) edge node[above] {$f^+$} (b)
	(c) edge node[above] {$g^-$} (d)
	(b) edge node[right] {$g^+$} (d)
	(a) edge node[left] {$f^-$} (c);
	\end{tikzpicture}
\end{center}

We denote by $\calD^b_{\bbC^*}(Y)$ the bounded $\bbC^*$-equivariant derived category of constructible sheaves of $\bbQ$-vector spaces on $Y$.
\begin{definition}
The \emph{hyperbolic localization functors} $\HL^{!*},\HL^{*!}:\calD^b_{\bbC^*}(Y) \ra \calD^b(F)$ are defined as 
\[\HL^{!*}:=(f^+)^!(g^+)^*\aand \HL^{*!}:=(f^-)^*(g^-)^!.\]
\end{definition}

Let $\pi^+:Y^+\ra F$ (resp. $\pi^-:Y^-\ra F$) be the projection defined by $\pi^+(y)= \lim_{t\ra 0} t\cdot y$ (resp. $\pi^-(y)= \lim_{t\ra \infty} t\cdot y$).

We recall the main theorem from \cite{BraHyperbolic}.
\begin{thm} \label{hlthm}Let $S \in \calD^b_{\bbC^*}(Y)$.
\begin{enumerate}
 \item We have natural isomorphisms \[\HL^{*!}(S)\cong  (\pi^-)_*(g^-)^!(S)\cong (\pi^+)_!(g^+)^*(S)\cong \HL^{!*}(S).\]
\item If $S$ is a semisimple perverse sheaf, then $\HL^{!*}(S)$ is the direct sum of shifts of semisimple perverse sheaves.
\end{enumerate}
\end{thm}

In view of \Cref{hlthm} we can simply denote by $\HL(S)$ either one of the hyperbolic localization functors $\HL^{!*}(S)$ and $\HL^{*!}(S)$.

For $i_Z:Z\hookrightarrow Y$ a closed embedding, we denote by $H^\bullet_Z(S):=H^\bullet(i_Z^!S)$ the cohomology with supports in $Z$ of $S$. We denote by $H^\bullet_c(S)$ the cohomology with compact supports of $S$. From \Cref{hlthm}(1), for any $S\in \calD^b_{\bbC^*}(Y)$ we have an isomorphism $H_{Y^{-}}(S)\cong H_c(S\mid_{Y^+})$.

We apply the previous theorem to our situation. Let $Y=\bar{\Gr_\lambda}$ be a Schubert variety where $\lambda\in X_+$. We have an action of $\affT$ on $Y$, hence any cocharacter $\eta\in \affX=X\oplus \bbZ d$ determines an action of a rank $1$ subtorus $\bbC^*$ on $Y$. We say that $z\in \sch{\lambda}$ is a fixed point for $\eta$ if it is a fixed point for the corresponding $\bbC^*$-action.

\begin{definition}
 We say that a cocharacter $\eta\in \affX$ is \emph{regular} if $\langle \eta,\alpha^\vee\rangle \neq 0$ for any $\alpha^\vee \in \affPhi^\vee$. We say that $\eta\in \affX$ is \emph{singular} if it is not regular.
\end{definition}

If we choose a regular cocharacter $\eta$, then its fixed points in $\bar{\Gr_\lambda}$ are exactly the points $L_\mu$ for $\mu \leq \lambda$. Therefore, for any $\mu\leq \lambda$ we have a hyperbolic localization functor
\[\HL_\mu^\eta: \calD^b_{\affT}(\sch{\lambda})\lra \calD^b(pt).\]

Let $\eta$ be regular and let $Y^+_{\mu,\eta}$ and $Y^-_{\mu,\eta}$ denote respectively the attractive and repulsive sets of $L_\mu$ in $\Gr^\vee$.

By \Cref{hlthm}, 
for any $\calF\in \calD^b_{\affT}(\sch{\lambda})$ we have 
\[\HL_\mu^\eta(\calF)\cong H^\bullet_{Y_{\mu,\eta}^-\cap \sch{\lambda}}(\calF) \cong H_c^\bullet(\calF|_{Y_{\mu,\eta}^+\cap \sch{\lambda}}).\]
Notice that for any $\calF\in \calD^b_{\affT}(\sch{\lambda})$ we have \[H^\bullet_{Y_{\mu,\eta}^-}(\calF)\cong H^\bullet_{Y_{\mu,\eta}^-\cap \bar{ \Gr_\lambda}}(\calF)\quad\text{and}\quad H_c^\bullet(\calF\mid_{Y_{\mu,\eta}^+})\cong H_c^\bullet(\calF\mid_{Y_{\mu,\eta}^+\cap \bar{ \Gr_\lambda}}).\]
Hence, the complex $\HL_\mu^\eta(\calF)$ does not depend on $\lambda$, i.e. it does not change if we regard $\calF$ as an object of $\calD^b_{\affT}(\sch{\lambda})$ or of $\calD^b_{\affT}(\sch{\lambda'})$, for $\lambda'>\lambda$.

\begin{remark}
	The category of $G^\vee[[t]]$-equivariant perverse sheaves is equivalent to the category of $G^\vee[[t]]\rtimes \bbC^*$-equivariant perverse sheaves (where $\bbC^*$ acts as loop rotations), and they are both equivalent to the category of constructible perverse sheaves with respect of the stratification of $G^\vee[[t]]$-orbits by \cite[Proposition A.1]{MVGeometric}.
	
	Moreover, recall from \cite[Proposition A.1]{BRNotes} that being equivariant should be regarded as a property of a perverse sheaf, rather than an additional structure on the sheaf. For this reason, we may consider $\IC_\lambda$ as $G^\vee[[t]]\rtimes \bbC^*$-equivariant objects.
\end{remark}
\begin{definition}
Let $\eta\in X\oplus \bbZ d$ be a regular cocharacter. Let $\lambda \in X_+$ and let $\mu\in X$ with $\mu\leq \lambda$. 
Then we define the following Laurent polynomials 
\[ \htil_{\mu,\lambda}^\eta(v) := \grdim\left(\HL_\mu^\eta( \IC_\lambda)\right),
\qquad h_{\mu,\lambda}^\eta(v):=\htil_{\mu,\lambda}^\eta(v)v^{\dim (Y_{\mu,\eta}^+\cap\bar{ \Gr_\lambda})}.\]
We call $h_{\mu,\lambda}^\eta(v)$ a $\eta$\emph{-Kazhdan--Lusztig polynomial}
and $\htil_{\mu,\lambda}^\eta(v)$ a \emph{renormalized $\eta$-Kazhdan--Lusztig polynomial}. 
\end{definition}

The $\eta$-Kazhdan--Lusztig polynomials can be thought of as a generalization of the ordinary Kazhdan--Lusztig polynomials $h_{\mu,\lambda}(v)$. In fact, we have the following proposition.

\begin{prop}\label{KLprop}
Let $\eta \in \affX$ be such that
$\langle \eta,\alpha^\vee \rangle >0$ for all $\alpha^\vee\in \affPhi^\vee_+$. Then $h_{\mu,\lambda}^\eta=h_{\mu,\lambda}$ for any $\mu,\lambda\in X_+$ with $\mu\leq \lambda$.
\end{prop}
\begin{proof}
 Let $\eta$ be as above. Then the attractive set of $L_\mu$ in $\sch{\lambda}$ relative to $\eta$ is precisely the Bruhat cell $\calC_\mu$, so we have $\dim(Y_{\mu,\eta}^+\cap\bar{ \Gr_\lambda})=2\langle \mu,\rho^\vee\rangle=\ell(\mu)$. 
 We need to compute
 \begin{equation}\label{KLpropEq}\grdim\left( \HL^\eta_\mu(\IC_{\lambda})\right)=\grdim\left( i_{\mu}^!(\IC_{\lambda}|_{\calC_\mu})\right)\end{equation}
where $i_{\mu}:\{ L_\mu\} \hookrightarrow \calC_\mu$ is the inclusion.
 
 Recall that the intersection cohomology $\IC_\lambda$ is constant along Bruhat cells, and that $\calC_\mu$ is an affine space of dimension $\ell(\mu)$. Therefore, we have \[i_{\mu}^!(\IC_\lambda|_{\calC_\mu})\cong i_{\mu}^*(\IC_\lambda|_{\calC_\mu})[-2\ell(\mu)]\cong i_{\mu}^*\IC_{\lambda}[-2\ell(\mu)]\]
 and we conclude since
 \[h_{\mu,\lambda}^\eta(v)=\grdim \HL_\mu^\eta(\IC_{\lambda})v^{\ell(\mu)}=\grdim (i_{\mu}^*\IC_{\lambda})v^{-2\ell(\mu)+\ell(\mu)}=h_{\mu,\lambda}(v),\]
 where the last equality follows from \eqref{chIC} and \eqref{KLbasis}.
\end{proof}

Although there are infinitely many real roots, a $\eta$ as in \Cref{KLprop} always exists. In fact, those are precisely the cocharacters $\eta = \lambda+Cd \in \affX= X\oplus \bbZ d$ such that 
\begin{equation}\label{KLchamber}
\lambda \in X_{++}\aand C>\langle \lambda,\beta^\vee\rangle\text{ for all }\beta^\vee \in \Phi_+^\vee
\end{equation}
where $X_{++}:=\{\lambda \in X \mid \langle \lambda ,\beta^\vee\rangle >0$ for any $\beta^\vee\in \Phi^\vee_+\}$.
\begin{definition}
 We call \emph{KL region} the subset of $\affX$ of the elements $\eta=\lambda+Cd$ satisfying \eqref{KLchamber}.
\end{definition}



 
\subsubsection{The MV Region}\label{MVsection}
There is another region of $\affX$ for which the hyperbolic localization functors have been thoroughly studied. 

\begin{definition}
 We call \emph{MV} (or \emph{Mirkovi\'c--Vilonen}) \emph{region} the subset 
 \[ \{\lambda + Cd\in X\oplus \bbZ d \mid \lambda \in X_{++} \text{ and }C=0\}\cu \affX\]
\end{definition}

In other words, a cocharacter $\eta$ in the MV region is a dominant cocharacter of $T^\vee$ which we think as a cocharacter of $\affT$ via the inclusion $T^\vee\ra \affT$.

The attractive and repulsive sets $Y_{\mu,\eta}^+$ and $Y_{\mu,\eta}^-$ for $\eta$ in the MV region are known as the \emph{semi-infinite orbits}.
In this case the hyperbolic localization functors are 
known as \emph{weight functors}: they correspond via the geometric Satake equivalence to the weight spaces in representation theory.

We recall some results about the weight functors from \cite{MVGeometric,BRNotes}.

\begin{thm}\label{weightfunctors}
Let $\calF\in \Perv_{G^\vee[[t]]}(\Gr^\vee)$, let $\eta\in \affX$ be in the MV region and let $\mu\in X$. Then
\begin{enumerate}
 \item We have $H^k(\HL_\mu^\eta(\calF))\cong H^k_c(\calF|_{Y_{\mu,\eta}^+})\cong H^k_{Y_{\mu,\eta}^-}(\calF)$ and they all vanish for $k\neq 2\langle \mu,\rho^\vee\rangle$,
 \item For any $k\in \bbZ$ we have
 $\displaystyle H^k(\calF)=\bigoplus_{\substack{\mu \in X\\2\langle \mu,\rho^\vee\rangle = k}} H^k_c(\calF|_{Y_{\mu,\eta}^+})$,
 \item\label{point3} For any $\lambda\in X_+$ we have $\sch{\lambda} \cap Y_{\mu,\eta}^+\neq \emptyset$ if and only if $\mu\leq \lambda$ in the Bruhat order. If this is the case, then $\Gr_\lambda \cap Y_{\mu,\eta}^+$ is of pure dimension $\langle \mu+\lambda,\rho^\vee\rangle$,
 \item\label{point4} We have
 \[H^{\langle \mu,2\rho^\vee\rangle}_c(\IC_\lambda\mid_{{Y_{\mu,\eta}^+}})=H_c^{\langle \mu+\lambda,2\rho^\vee\rangle}(\Gr_\lambda \cap Y_{\mu,\eta}^+).\]
\end{enumerate}
\end{thm}

The irreducible components of $\Gr_\lambda \cap Y_{\mu,\eta}^+$ are called \emph{MV cycles}. It follows that the MV cycles provide a basis of $\HL_\mu^\eta(\IC_{\lambda})$. It is immediate to deduce from \Cref{weightfunctors} the following consequence for the $\eta$-KL polynomials.

\begin{corollary}\label{KLinMV}
Let $\eta$ be in the MV region. Let $\lambda \in X_+$ and $\mu\in X$ be such that $\mu\leq \lambda$. Then
\[\htil_{\mu,\lambda}^\eta(v)=d_{\mu,\lambda}v^{-2\langle\mu,2\rho^\vee \rangle}\aand h_{\mu,\lambda}^\eta(v)=d_{\mu,\lambda}v^{\langle\lambda-\mu ,\rho^\vee\rangle},\]
where $d_{\mu,\lambda}$ is the number of irreducible components of $\Gr_\lambda \cap Y_{\mu,\eta}^+$.
\end{corollary}


\subsubsection{Walls and Chambers}\label{crossingSection}

We have discussed the hyperbolic localization functor for $\eta$ either in the MV or in the KL region.
Our next goal is to study the $\eta$-Kazhdan--Lusztig polynomials for a general $\eta\in \affX$.

Let $\affX_\bbR := \affX \otimes_\bbZ \bbR \cu \frk^\vee$.

\begin{definition}\label{wall}
We call \emph{wall} a hyperplane of the form 
\[H_{\alpha^\vee}:=\{ \eta \in \affX_\bbR \mid \langle \eta,\alpha^\vee \rangle =0\}\cu \affX_\bbR\]
for $\alpha^\vee\in \affPhi^\vee$.
\end{definition}

\begin{example}
	Assume that $\Phi=\{\pm \alpha\}$ is a root system of type $A_1$. We draw in blue the walls of $\affX_\bbR$ in the basis $(\varpi,d)$. We color in red the KL region and in green the MV region.
	
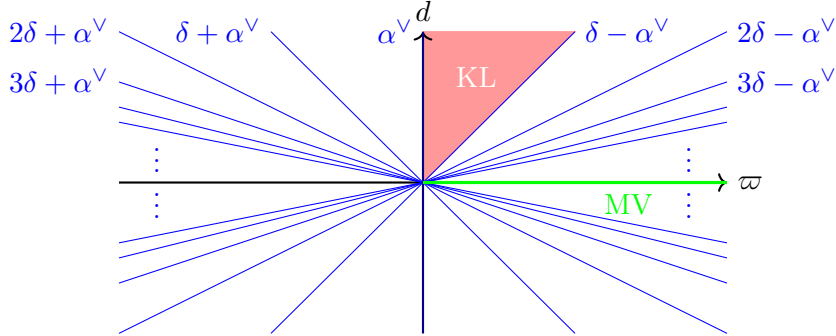
\begin{center}
	\begin{tikzpicture}
	\filldraw[red!40,fill=red!40] (0,0) -- (2,2) -- (0,2) -- cycle;
	\node[white] at (0.7,1.4) {KL};
	\draw[thick,->] (-4,0) -- (4,0) node[right] {$\varpi$};
	\draw[thick,->] (0,-2) -- (0,2) node[above] {$d$};
	\draw[blue] (0,-2) -- (0,2) node[left] {$\alpha^\vee$};
	\draw[blue] (2,-2) -- (-2,2) node[left] {$\delta+\alpha^\vee$};
	\draw[blue] (-2,-2) -- (2,2) node[right] {$\delta-\alpha^\vee$};
	\draw[blue] (-4,-2) -- (4,2) node[right] {$2\delta-\alpha^\vee$};
	\draw[blue] (4,-2) -- (-4,2) node[left] {$2\delta+\alpha^\vee$};
	\draw[blue] (-4,-2*2/3) -- (4,2*2/3) node[right] {$3\delta-\alpha^\vee$};
	\draw[blue] (4,-2*2/3) -- (-4,2*2/3) node[left] {$3\delta+\alpha^\vee$};
	\draw[blue] (-4,-2*2/4) -- (4,2*2/4);
	\draw[blue] (-4,-2*2/5) -- (4,2*2/5);
	\draw[blue] (4,-2*2/4) -- (-4,2*2/4);
	\draw[blue] (4,-2*2/5) -- (-4,2*2/5);
	\draw[blue] node at (3.5,0.4) {$\vdots$};
	\draw[blue] node at (3.5,-0.2) {$\vdots$};
	\draw[blue] node at (-3.5,0.4) {$\vdots$};
	\draw[blue] node at (-3.5,-0.2) {$\vdots$};
	\draw[very thick, green] (0,0) -- node[below,xshift=20] {MV} (4,0);
	\end{tikzpicture}
	\end{center}
\end{example}
	If $N$ is a positive integer and $\eta$ is a cocharacter, the hyperbolic localization functors for $\eta$ and $N\eta$ coincide. This allows us to extend the definition of the hyperbolic localization functors to $\affX_\bbQ:=\affX\otimes_{\bbZ}\bbQ$.
	
	\begin{definition}
		Let $\eta \in \affX_\bbQ$ and let $N\in \bbZ_{>0}$ be such that $N\eta \in \affX$. For $\mu\in X$, we set $\HL^\eta:=\HL^{N\eta}$ and $Y_{\mu,\eta}^\pm :=Y_{\mu,N\eta}^\pm$. 
		We also abuse notation and write 
		\[ \lim_{t\ra \star}\eta(t)\cdot z:=\lim_{t\ra \star}(N\eta)(t)\cdot z\]
		for $z\in \Gr^\vee$ and $\star\in \{0,\infty\}$. (This is well-defined, even though $\eta(t)$ is not.) We also say that $z$ is a fixed point for $\eta$ if it is a fixed point for $N\eta$.
	\end{definition}

\begin{definition}\label{chamber}
	For $\lambda\in X_+$, we denote by $\affPhi^\vee(\lambda)$ the set of all the labels present in the moment graph $\Gamma_\lambda$ described in \Cref{MGsection}. We say that a wall $H_{\alpha^\vee}$ is a $\lambda$-wall if $\alpha^\vee\in \affPhi^\vee(\lambda)$.

We call \emph{$\lambda$-chamber} (or simply chamber, if $\lambda$ is clear from the context)  the intersection of $\affX_\bbQ$ with a connected component of 
\[\affX_\bbR \setminus \bigcup_{\alpha^\vee\in \affPhi^\vee(\lambda)}H_{\alpha^\vee}.\]
We say that two chambers are \emph{adjacent} if they are separated by a single $\lambda$-wall. 

We call \emph{KL chamber} the unique chamber containing the KL region and  \emph{MV chamber} the unique chamber containing the MV region.
\end{definition}

Notice that each connected component of $\affX_\bbR\setminus \bigcup H_{\alpha^\vee}$ is an open cone, so it always contains an element of $\affX_\bbQ$, so there is a bijection between chambers and connected components.

\begin{remark}
	For most $\lambda\in X_+$ (i.e., if $\lambda$ is not too small) the cocharacters $\mu\in \affX$ in the KL chamber are precisely the cocharacters in the KL region.
	
	It will follow from \Cref{HLinChamber} that
\Cref{KLprop} and \Cref{KLinMV} still hold for any $\eta$ in the KL chamber or in the MV chamber, respectively.
\end{remark}

	As the following example shows, the KL chamber depends on $\lambda$ and it does not always coincide with the KL region.

\begin{example}
	Let $\Phi$ the root system of type $A_2$ as in \Cref{A2}. Let $\lambda=\varpi_1$ be the first fundamental weight. The moment graph $\Gamma_{\varpi_1}$ is the following graph.
	
		
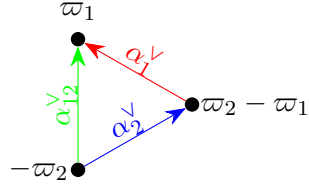
\begin{center}
		\begin{tikzpicture}
		\tikzstyle{every node}=[draw,circle,fill=black,minimum size=5pt,
		inner sep=0pt]	
		\draw (120:1) node (1)[label=above:$\varpi_1$] {};
		\draw (0:1) node (2)[label=right:$\varpi_2-\varpi_1$] {};
		\draw (-120:1) node (3)[label=left:$-\varpi_2$] {};
		\tikzstyle{every node}=[inner sep=0pt]
		\draw[-{Stealth[scale=1.6]},red] (2) -- node[above,sloped] {$\alpha_1^\vee$} (1);
		\draw[-{Stealth[scale=1.6]},blue] (3) -- node[above,sloped] {$\alpha_2^\vee$} (2);
		\draw[-{Stealth[scale=1.6]},green] (3) -- node[above,sloped] {$\alpha_{12}^\vee$} (1);
		\end{tikzpicture}
	\end{center}

	Consider the family of cocharacters $\eta(t):=\mu + td$, with $\mu \in X_{++}$ and $t\in \bbQ$. Then $\eta(0)$ belongs to the MV region while $\eta(t)$ belongs to the KL region if $t$ is large enough. Moreover, the family $\eta(t)$ does not cross any of the $\lambda$-walls (which are $H_{\alpha_1^\vee}$, $H_{\alpha_2^\vee}$ and $H_{\alpha_1^\vee+\alpha_2^\vee}$) meaning that it is contained in a single $\lambda$-chamber. In particular, in this case the MV chamber and the KL chamber coincide.
\end{example}

\subsubsection{Chambers and Hyperbolic Localization}

We recall Sumihiro's theorem on torus actions.
\begin{thm}\label{Sumihiro}
	Let $X$ be a normal variety equipped with an action of an algebraic torus $T$. Then each point of $X$ has a $T$-stable affine open neighborhood $U$.
	
	Moreover, there exists a $T$-equivariant closed embedding $U\hookrightarrow V$  where $V$ is a finite dimensional $T$-module.
\end{thm}
\begin{proof}
This is \cite[Cor. 2]{SumEquivariant}. The second statement follows as a corollary as explained, for example, in \cite[Prop. 2.2.5]{BriLinearization}.
\end{proof}

Since the Schubert varieties $\bar{ \Gr_\lambda}$ are normal, Sumihiro's theorem applies to our situation. We are going to use it to determine how the hyperbolic localization depends on the cocharacter $\eta$.
In the following Lemma we show that the hyperbolic localization of $\sch{\lambda}$ can only change if we cross a $\lambda$-wall.

\begin{lemma}\label{etainthemiddle}
	Let $\eta,\eta'\in \affX_\bbQ$. Assume there exists $z\in (Y_{\mu,\eta}^+\cap \sch{\lambda})\setminus Y_{\mu,\eta'}^+$. For $\tau\in [0,1]_\bbQ:=[0,1]\cap \bbQ$ consider the cocharacter $\eta_\tau:=(1-\tau)\eta+\tau\eta'\in \affX_\bbQ$. Then, there exists $\tau$ such that 
	\[ \lim_{t\ra 0}\eta_\tau(t)\cdot z\in Y_{\mu,\eta}^+\setminus \{ L_\mu\}\]
	and $\eta_\tau$ belongs to a $\lambda$-wall.  
\end{lemma}
\begin{proof}
	By \Cref{Sumihiro}, we can find a $\affT$-stable  open neighborhood $U$ of $L_\mu$ and a closed embedding $\phi:U\hookrightarrow V$ into a finite dimensional $\affT$-module $V$ such that $\phi(L_\mu)=0$. 
	Let \[V=\bigoplus_{\gamma \in P(V)} V_{\gamma}\] be the decomposition of $V$ into weight spaces, where $P(V):=\{\gamma \in \affX^\vee \mid V_\gamma \neq 0\}.$  For $v \in V$ we write $v = \sum v_\gamma$ with $v_\gamma \in V_\gamma$. For any $\zeta\in \affX$ we have 
	$\zeta(t)\cdot v = \sum t^{\langle \zeta,\gamma\rangle} v_\gamma$.

Since $U$ is $\affT$-stable and $z$ is attracted to $L_\mu$ by $\eta$ we have $z\in U$.	
Let $P(V,z):=\{\gamma \in P(V) \mid \phi(z)_\gamma\neq 0\}$.
	Since $\lim_{t\ra 0} \eta(t)\cdot \phi(z)=0$ we have $\langle \eta,\gamma\rangle >0$ for every $\gamma \in P(V,z)$. Since $\lim_{t\ra 0} \eta'(t)\cdot \phi(z)\neq 0$, there exists $\gamma\in P(V,z)$ such that $\langle \eta',\gamma\rangle\leq 0$.
	
	The set $P(V,z)$ is finite, so we can find $\tau\in [0,1]_\bbQ$ such that $\langle \eta_\tau,\gamma\rangle =0$ for some $\gamma \in P(V,z)$ and  $\langle \eta_\tau,\gamma\rangle \geq 0$ for all $\gamma \in P(V,z)$. Explicitly, we take
	\[\tau =\min_{\substack{\gamma\in P(V,z)\\ \langle \eta',\gamma\rangle \leq 0}}\frac{\langle\eta,\gamma\rangle}{\langle\eta,\gamma\rangle-\langle\eta',\gamma\rangle}.\]
	It follows that
	\[\lim_{t\ra 0} \phi(\eta_\tau(t)\cdot z)=\lim_{t\ra 0} \eta_\tau(t)\cdot \phi(z)=\sum_{\substack{\gamma \in P(V,z)\\ \langle \eta_\tau,\gamma\rangle =0}} \phi(z)_\gamma\in V\setminus \{0\}.\]
	Since $\phi$ is a closed embedding, we deduce that \[y:=\lim_{t\ra 0} \eta_\tau(t)\cdot z \in U \cap (Y_{\mu,\eta}^+\setminus \{L_\mu\}).\] 
	
	Notice that $y$ is a $\eta_\tau$-fixed point. If $\eta_\tau$ does not belong to any $\lambda$-wall, by \eqref{cell} we see that the only $\eta_\tau$-fixed points  are of the form $L_\nu$ for $\nu\leq \lambda$. However, this leads to a contradiction since $L_\nu\not \in Y_{\mu,\eta}^+$ for $\nu\neq \mu$. 
\end{proof}

\begin{prop}\label{HLinChamber}
	Let $\eta,\eta'\in \affX_\bbQ$ be cocharacters belonging to the same $\lambda$-chamber. Then $Y_{\mu,\eta}^\pm\cap \bar{\Gr_\lambda}=Y_{\mu,\eta'}^\pm\cap \bar{\Gr_\lambda}$. In particular, we have $\HL_\mu^\eta=\HL_\mu^{\eta'}$ on $\calD^b_{\affT}(\sch{\lambda})$.
\end{prop}
\begin{proof}
	Assume that there exists $z\in (Y_{\mu,\eta}^+\cap \sch{\lambda})\setminus Y_{\mu,\eta'}^+$. Then there exists $\eta_\tau$ as in \Cref{etainthemiddle} lying on a $\lambda$-wall. But this is impossible since $\lambda$-chambers are convex sets. This shows that $Y_{\mu,\eta}^+\cap \bar{\Gr_\lambda} \subset Y_{\mu,\eta'}^+\cap \bar{\Gr_\lambda}$. The proof of the opposite inclusion and of $Y_{\mu,\eta}^-\cap \bar{\Gr_\lambda}=Y_{\mu,\eta'}^-\cap \bar{\Gr_\lambda}$ are similar.
\end{proof}
	 
%
%
%
%

In general, the attractive and repulsive sets on $\bar{\Gr_\lambda}$ can only change when we cross a $\lambda$-wall, so $\HL^\eta_\mu$ only depends on the $\lambda$-chamber in which $\eta$ lies.

\subsubsection{Wall Crossing in Hyperbolic Localization}\label{sec:WallCrossing}

We now want to understand how the hyperbolic localization functor varies between adjacent chambers. Let $C_1$ and $C_2$ be two adjacent $\lambda$-chambers which border each other on a common $\lambda$-wall $H_{\alpha^\vee}$, i.e. we assume $\emptyset\neq \bar{C_1}\cap \bar{C_2} \cu H_{\alpha^\vee}$ and $C_1\cap C_2 = \emptyset$.

Choose $\eta_1\in C_1$ and $\eta_2\in C_2$. 
We assume 
\[\langle \eta_1, \alpha^\vee \rangle<0\aand \langle \eta_2, \alpha^\vee\rangle >0.\]

Then we take $\eta_s\in \affX_\bbQ$ as the point at the intersection  between $H_{\alpha}^\vee$ and the line passing through $\eta_1$ and $\eta_2$ (i.e., we have $\eta_s=(1-r)\eta_1+r\eta_2$ for some $r\in [0,1]_\bbQ$), as in the following figure.

	
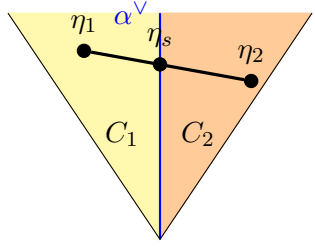
\begin{center}
	\begin{tikzpicture}

	\filldraw[orange!40,fill=orange!40] (0,0) -- (2,3) -- (0,3) -- cycle;
	\node at (0.5,1.4) {$C_2$};
	\filldraw[yellow!40,fill=yellow!40] (0,0) -- (-2,3) -- (0,3) -- cycle;
\node at (-0.5,1.4) {$C_1$};	
	\draw[blue, thick] (0,0) -- (0,3) node[left] {$\alpha^\vee$};
	\draw (0,0) -- (2,3);
	\draw (0,0) -- (-2,3);
	\node[draw,circle,fill=black,minimum size=5pt,
	inner sep=0pt] (a) at (-1,2.5) [label=above:$\eta_1$] {};
	\node[draw,circle,fill=black,minimum size=5pt,
inner sep=0pt] (b) at (1.2,2.1) [label=above:$\eta_2$] {};	
	\node[draw,circle,fill=black,minimum size=5pt,
inner sep=0pt]  at (0,2.32) [label=above:$\eta_s$] {};
	\draw[very thick] (a) -- (b);

	\end{tikzpicture}
\end{center}


The cocharacter $\eta_s$ is singular but we can still easily compute the $\eta_s$-fixed points, which we denote by  $\sch{\lambda}^s\subset\sch{\lambda}$. 

\begin{lemma}\label{ccetas}
	Let $\mu\leq \lambda$ and let $\frC_\mu$ be the connected component of $L_\mu$ in $\sch{\lambda}^s$. Then one of the following two statements holds.
\begin{enumerate}
	\item $s_{\alpha^\vee}(\mu)\not \leq \lambda$ or $s_{\alpha^\vee}(\mu)=\mu$, and $\frC_\mu=\{L_\mu\}$.
	\item $\mu\neq s_{\alpha^\vee}(\mu) \leq \lambda$ and $\frC_\mu=\{L_\mu,L_{s_{\alpha^\vee}(\mu)}\} \cup \calO_\mu$, where $\calO_\mu$ is the one-dimensional $\affT$-orbit between $\mu$ and $s_{\alpha^\vee}(\mu)$.

\end{enumerate}	
\end{lemma}
\begin{proof}
	
Let $p\in \Gr^\vee$ be a point fixed by $\eta_s$. We have $p\in \calC_\nu$ for some $\nu \in X$ and, by \eqref{cell}, we can write \[p=\prod_{\substack{\gamma^\vee\in \affPhi^\vee\\ s_{\gamma^\vee}(\nu)<\nu}}u_{\gamma^\vee}(x_{\gamma^\vee})\cdot L_\nu\] for some $x_{\gamma^\vee}\in \bbC$. Since $\eta_s(z)\cdot p=p$ we see that $x_{\gamma^\vee}=0$ unless $\langle \eta_s,\gamma^\vee\rangle =0$, which only occurs if $\gamma^\vee=\alpha^\vee$. It follows that $p$ is a fixed point if and only if $p=L_\nu$ for some $\nu \in X$ or, as in \Cref{onedimorbits}, if $p$ belongs to the one-dimensional $\affT$-orbit $\calO_{\nu}$ between $\nu$ and $s_{\alpha^\vee}(\nu)$.
 for any $\nu\in X$ with $s_{\alpha^\vee}(\nu)\neq \nu$. It follows that if $\nu=s_{\alpha^\vee}(\nu)$, then $L_\nu$ is an isolated point in $\sch{\lambda}^s$, and $\frC_\nu=\{L_\nu\}$.

	If $s_{\alpha^\vee}(\nu)\not \leq \lambda$, then $\nu<s_{\alpha^\vee}(\nu)$ and $\calO_{\nu}$ is contained in the Bruhat cell $\calC_{s_{\alpha^\vee(\nu)}}$, so $\calO_\nu \cap \sch{\lambda}=\emptyset$ and $\frC_\nu=\bar{\calO_\nu}\cap \sch{\lambda}=\{L_\nu\}$.
	
	If $\nu \neq s_{\alpha^\vee}(\nu) \leq \lambda$, then $\calO_{\nu}$ is contained in $\sch{\lambda}$ and $\frC_\nu=\bar{\calO_\nu}=\{L_\nu,L_{s_{\alpha^\vee}(\nu)}\}\cup \calO_{\nu}$.
\end{proof}

We start by considering the first case of \Cref{ccetas}: if $\frC=\{L_\mu\}$, then the hyperbolic localization does not change when we cross the wall $H_{\alpha^\vee}$, i.e. we have
$\HL_\mu^{\eta_1}= \HL_\mu^{\eta_2}$.

\begin{lemma}\label{HLadjacentgood}
	Assume $s_{\alpha^\vee}(\mu)\not \leq \lambda$ or $s_{\alpha^\vee}(\mu)=\mu$. Then we have $Y_{\mu,\eta_1}^\pm\cap \sch{\lambda}=Y_{\mu,\eta_2}^\pm\cap \sch{\lambda}$.
\end{lemma}
\begin{proof}
	
	Assume that there exists $z\in (Y_{\mu,\eta_1}^+\cap \sch{\lambda})\setminus Y_{\mu,\eta_2}^+$ and let  $\eta_\tau:=(1-\tau)\eta_1+\tau\eta_2$ be as in \Cref{etainthemiddle}. We have $\eta_\tau\in H_{\alpha^\vee}$ since this is the unique $\lambda$-wall separating $\eta_1$ and $\eta_2$. It follows that $\eta_\tau=\eta_s$ and that
\[y=\lim_{t\ra 0}\eta_{s}(t)\cdot z\in Y_{\mu,\eta_1}^+\setminus \{L_\mu\}.\]

Let $N\in \bbZ_{>0}$ be such that $N\eta_1\in \affX$.
	For any $t\in \bbC^*$ the element $(N\eta_1)(t)\cdot y$ is a fixed point for $\eta_s$. This means that $y$ belongs to the connected component  $\frC_\mu$ of $L_\mu$, leading to a contradiction since, by \Cref{ccetas}, $L_\mu$ should be an isolated $\eta_s$-fixed point. It follows that $Y_{\mu,\eta_1}^+\cap \sch{\lambda}\subset Y_{\mu,\eta_2}^+\cap \sch{\lambda}$. The opposite inclusion and the equality $Y_{\mu,\eta_1}^-\cap \sch{\lambda}= Y_{\mu,\eta_2}^-\cap \sch{\lambda}$ are proven similarly. 
\end{proof}

The second case in \Cref{ccetas}, when $\dim \frC_\mu =1$, is more interesting and deserves more work.
For the rest of this section, we fix $\mu,\nu \in X$ with $\nu=s_{\alpha^\vee}(\mu)$ such that $\mu < \nu\leq \lambda$ in the Bruhat order. Let $\calO$ be the one-dimensional $\affT$-orbit between $L_\mu$ and $L_\nu$.
The connected component $\frC_\mu=\{L_\mu,L_{s_{\alpha^\vee}(\mu)}\} \cup \calO=\barO$ is isomorphic to the projective line $\bbP^1(\bbC)$. To make the notation lighter we set $\HL_{(-)}^i:=\HL_{(-)}^{\eta_i}$ and $Y_{(-),i}^+:=Y_{(-),\eta_i}^+$ for $i\in \{1,2,s\}$. 	Let $\pi: \sch{\lambda} \ra \sch{\lambda}^s$ denote the projection induced by $\lim_{t\ra 0} \eta_s(t)$.

\begin{lemma}

 We have
\begin{alignat*}{2}
 Y_{\mu,1}^+ \cap \sch{\lambda} &=\pi^{-1}\left(\barO\setminus \{L_\nu\}\right),\qquad && Y_{\nu,1}^+\cap \sch{\lambda} =\pi^{-1}\left(L_\nu\right),\\ 
Y_{\mu,2}^+\cap \sch{\lambda}&=\pi^{-1}\left( L_\mu\right),\qquad && Y_{\nu,2}^+\cap \sch{\lambda} =\pi^{-1}\left( \barO \setminus \{L_\mu\}\right).
\end{alignat*}
\end{lemma}
\begin{proof}
	
	Notice that $\barO\setminus \{L_\nu\}=Y_{\mu,1}^+\cap \barO$ and $\{L_\mu\}=Y_{\mu,2}^+\cap \barO$. Similarly, we have $\{L_\nu\}=Y_{\nu,1}^+\cap \barO$ and $\barO\setminus \{L_\nu\}=Y_{\mu,2}^+\cap \barO$. Thus, we need to show, for $i\in\{1,2\}$ and $\zeta \in \{\mu,\nu\}$, that
	\[ Y_{\zeta,i}^+\cap \sch{\lambda}=\pi^{-1}(Y_{\zeta,i}^+\cap \barO).\]
	
	We first show that 
	\begin{equation}\label{pizeta}\pi(Y_{\zeta,i}^+\cap \sch{\lambda})\subset Y_{\zeta,i}^+\end{equation} for any weight $\zeta\in X$. Take $z \in Y_{\zeta,i}^+\cap \sch{\lambda}$. If $\pi(z)=L_\zeta$, then clearly $\pi(z)\in Y_{\zeta,i}^+$. 
	Assume now that $\pi(z)\neq L_\zeta$, i.e. that
	$z \in Y_{\zeta,i}^+\setminus Y_{\zeta,s}^+$. By \Cref{etainthemiddle}, we can find $\eta_\tau$ in the segment between $\eta_i$ and $\eta_s$ such that $\lim_{t\ra 0}\eta_\tau(t)\cdot z\in Y_{\zeta,i}^+\setminus \{L_\zeta\}$ and such that $\eta_\tau$ belongs to a $\lambda$-wall. However, the only $\lambda$-wall between $\eta_i$ and $\eta_s$ is $H_{\alpha^\vee}$, on which also $\eta_s$ lies.  It follows that $\eta_{\tau}=\eta_s$ and that \[\pi(z)=\lim_{t\ra 0}\eta_s(t)\cdot z\in Y_{\zeta,i}^+,\]
showing \eqref{pizeta}.

	 Let $N\in \bbZ_{>0}$ be such that  $N\eta_i\in \affX$.
For every $t\in \bbC^*$ we have $(N\eta_i)(t)\cdot \pi(z)\in \sch{\lambda}^s$, so $\pi(z)$ belongs to $\frC_\mu=\barO$.
	This shows $Y_{\mu,i}^+\subset \pi^{-1}(Y_{\mu,i}^+\cap \barO)$. The other inclusion also follows since if $z\not\in Y_{\mu,i}^+$, then $z\in Y_{\zeta,i}^+\subset \pi^{-1}(Y_{\zeta,i}^+)$ for some other $\zeta\neq \mu$.
	\end{proof}

\begin{corollary}\label{HLeta12}
	Let $\calF \in \Perv_{G^\vee[[t]]}(\Gr^\vee)$ and $\calG:=\HL_{\barO}^{s}(\calF)$. Let $i_\mu:\{L_\mu\}\hookrightarrow \Gr^\vee$ denote the inclusion. We have
	\begin{align*}
	\HL_{\mu}^1(\calF) \cong H_c^\bullet(\calG\mid_{\barO\setminus \{ L_\nu\}}),\qquad& \HL_{\nu}^1(\calF) \cong i_\nu^*\calG,
	\\
	\HL_{\mu}^2(\calF) \cong i_\mu^*\calG, \qquad& \HL_{\mu}^2(\calF)\cong H_c^\bullet(\calG\mid_{\barO\setminus \{ L_\mu\}}).
	\end{align*}
\end{corollary}

\begin{proof}
	The isomorphisms concerning $\HL^1_\mu$ and $\HL^1_\nu$ follow from proper base change applied to the following Cartesian diagrams.
\[
		\begin{tikzpicture}
		\node (a) at (0,2) {$Y_{\mu,1}^+$};
		\node (b) at (2,2) {$Y_{\barO,s}^+$};
		\node (c) at (0,0) {$\barO\setminus \{L_\nu\}$};
		\node (d) at (2,0) {$\barO$};
		\path[right hook->] 
		(a) edge (b)
		(c) edge (d);
		\path[->] 
		(b) edge node[right] {$\pi$} (d)
		(a) edge node[right] {$\pi$} (c);
		\end{tikzpicture}\qquad\qquad
		\begin{tikzpicture}
		\node (a) at (0,2) {$Y_{\nu,1}^+$};
		\node (b) at (2,2) {$Y_{\barO,s}^+$};
		\node (c) at (0,0) {$\{ L_\nu\}$};
		\node (d) at (2,0) {$\barO$};
		\path[right hook->] 
		(a) edge (b)
		(c) edge node[above] {$i_\nu$}(d);
		\path[->] 
		(b) edge node[right] {$\pi$} (d)
		(a) edge node [right] {$\pi$} (c);
		\end{tikzpicture}\]
		The proof of the other isomorphisms is analogous.
\end{proof}


We now compute the hyperbolic localization functor for $\eta_s$.
For $p(v)=\sum p_i v^i\in \bbZ[v,v^{-1}]$ we use $p(v)\cdot \calF$ to denote $\bigoplus \calF[i]^{\oplus p_i}$. We denote by $\undQ_{\barO}$ the constant sheaf on $\barO$ and by $\undQ_{\mu}$ the skyscraper sheaf on $L_\mu$.

\begin{lemma}\label{HLetas}
Let $\calF \in \Perv_{G^\vee[[t]]}(\Gr^\vee)$. Then there exist $p(v),p'(v)\in \bbZ_{\geq 0}[v,v^{-1}]$ such that
 \[\HL_{\bar{\calO}}^{s}(\calF)\cong p(v)\cdot \undQ_{\barO}\oplus p'(v)\cdot \undQ_\mu.\]
\end{lemma}
\begin{proof}
Recall the root subgroup $U_{\alpha^\vee}$ from \Cref{secRootSG}. Let $G_{\alpha^\vee}$ be the group generated by $U_{\alpha^\vee}$ and $U_{-\alpha^\vee}$.
The group $G_{\alpha^\vee}$ is isomorphic to $SL_2$ or $PGL_2$. Then $B_{\alpha^\vee}= G_{\alpha^\vee}\cap \Iw$ is a Borel subgroup of $G_{\alpha^\vee}$ containing $U_{\alpha^\vee}$.

If $N\in \bbZ_{>0}$ is such that $N\eta_s\in \affX$, then the image of $N\eta_s$ commutes with $G_{\alpha^\vee}$. Since $\calF$ is $B_{\alpha^\vee}$-equivariant, by definition also $\HL_{\barO}^{s}(\calF)$ is $B_{\alpha^\vee}$-equivariant. The claim now follows by \Cref{hlthm}(2) since $\undQ_{\barO}$ and $\undQ_\mu$ are the only simple $B_{\alpha^\vee}$-equivariant perverse sheaves on $\bar{\calO}\cong \bbP^1(\bbC)$. 
\end{proof}

We are finally ready to compute how the $\eta$-KL polynomials change between cocharacters $\eta_1$ and $\eta_2$ in adjacent chambers.

\begin{prop}\label{crossing}
Let $\nu=s_{\alpha^\vee}(\mu)$ and assume that $\mu<\nu\leq \lambda$.
	We have
	\begin{align*} \htil_{\nu,\lambda}^{\eta_2}(v) =&\; v^{-2}\cdot \htil_{\nu,\lambda}^{\eta_1}(v)\\
 \htil_{\mu,\lambda}^{\eta_2}(v) =&\; \htil_{\mu,\lambda}^{\eta_1}(v)+ (1-v^{-2}) \cdot \htil_{\nu,\lambda}^{\eta_1}(v)\end{align*}
\end{prop}
\begin{proof}
	From \Cref{HLetas} we know that $\HL_{\barO}^s(\IC_\lambda)\cong p(v)\cdot \undQ_{\barO} \oplus p'(v) \cdot \undQ_{\mu}$ for some $p,p' \in \bbZ[v,v^{-1}]$. Recall that $\barO\setminus \{L_\mu\}\cong \barO\setminus \{L_\nu\}\cong \bbC$ and that $H^\bullet_c(\bbC)=\bbQ[-2]$. From \Cref{HLeta12} it follows that
	\begin{align*}
	\HL_\mu^1(\IC_\lambda)=\left(v^{-2}p(v)+p'(v)\right)\cdot \bbQ, \qquad & \HL_\nu^1(\IC_\lambda)=p(v)\cdot \bbQ,\\
	\HL_\mu^2(\IC_\lambda)=\left(p(v)+p'(v)\right)\cdot \bbQ, \qquad & \HL_\nu^2(\IC_\lambda)=v^{-2}p(v) \cdot \bbQ
	\end{align*}
	and the claim follows.
\end{proof}

\begin{example}	
	Let $\lambda=\alpha_1+\alpha_2$ as in \Cref{A2}.
	 There are exactly three $\lambda$-walls separating the MV chamber from the KL chamber: $H_{\delta-\alpha_1^\vee}$, $H_{\delta-\alpha_2^\vee}$ and $H_{\delta-\alpha_1^\vee-\alpha_2^\vee}$.
	
	Consider the family of cocharacters \[\eta:\bbQ_{\geq 0}\lra \affX_\bbQ,\qquad \eta(t)=A_1\varpi_1+A_2\varpi_2+td.\] 
	with $A_1,A_2\in \bbQ$ such that $0<A_1< A_2$. Then $\eta(0)$ is in the MV region, $\eta(t)$ is in the KL chamber for $t>A_1+A_2$ and we cross the three walls $H_{\delta-\alpha_1^\vee}$, $H_{\delta-\alpha_2^\vee}$ and $H_{\delta-\alpha_1^\vee-\alpha_2^\vee}$ at $t=A_1$, $t=A_2$ and $t=A_1+A_2$, respectively.
	We can choose arbitrarily $\eta_0\in \eta([0,A_1))$, $\eta_1\in \eta((A_1,A_2))$, $\eta_2\in \eta((A_2,A_1+A_2))$ and $\eta_3\ \in\eta((A_1+A_2,\infty))$, so that they lie in adjacent chambers.
	
	 Then, thanks to \Cref{crossing}, we can compute the related renormalized $\eta$-KL polynomials in any chamber as follows.
	 
\begin{center}
\begin{tabular}{c|c|c|c}
	$\htil^{\eta_0}_{\mu,\alpha_{12}}(v)$ & 	$\htil^{\eta_1}_{\mu,\alpha_{12}}(v)$ & 	$\htil^{\eta_2}_{\lambda,\alpha_{12}}(v)$ & 	$\htil^{\eta_3}_{\lambda,\alpha_{12}}(v)$\\
\hline
	\begin{tikzpicture}[scale=0.7]
	\tikzstyle{every node}=[draw,circle,fill=black,minimum size=5pt,
	inner sep=0pt]
	\draw (0,0) node [label=above:$2$] {};
	\draw (30:2) node [label=above:$v^{-2}$] {};
	\draw (90:2) node [label=above:$v^{-4}$] {};
	\draw (150:2) node [label=above:$v^{-2}$] {};
	\draw (-30:2) node [label=above:$v^2$] {};
	\draw (-90:2) node [label=below:$v^4$] {};
	\draw (-150:2) node [label=above:$v^2$] {};
	\end{tikzpicture} &
		\begin{tikzpicture}[scale=0.7]
	\tikzstyle{every node}=[draw,circle,fill=black,minimum size=5pt, inner sep=0pt]
	\draw (0,0) node (a) [label=above:$1+v^2$] {};
	\draw (30:2) node [label=above:$v^{-2}$] {};
	\draw (90:2) node [label=above:$v^{-4}$] {};
	\draw (150:2) node [label=above:$v^{-2}$] {};
	\draw (-30:2) node [label=right:$v^2$] {};
	\draw (-90:2) node [label=below:$v^4$] {};
	\draw (-150:2) node (b) [label=left:$1$] {};
	\tikzstyle{every node}=[]
	\draw[->,thick, red] (a) -- node[above,sloped] {$\delta-\alpha_1^\vee$} (b) [label=above:$s_\alpha^\vee$];
	\end{tikzpicture}
	&
	\begin{tikzpicture}[scale=0.7]
	\tikzstyle{every node}=[draw,circle,fill=black,minimum size=5pt, inner sep=0pt]
	\draw (0,0) node(a) [label=above:$2v^2$] {};
	\draw (30:2) node [label=above:$v^{-2}$] {};
	\draw (90:2) node [label=above:$v^{-4}$] {};
	\draw (150:2) node [label=above:$v^{-2}$] {};
	\draw (-30:2) node(b) [label=right:$1$] {};
	\draw (-90:2) node [label=below:$v^4$] {};
	\draw (-150:2) node [label=left:$1$] {};
	\tikzstyle{every node}=[]
	\draw[->,thick, red] (a) -- node[above,sloped] {$\delta-\alpha_2^\vee$} (b);
	\end{tikzpicture}
	&
	\begin{tikzpicture}[scale=0.7]
	\draw (0,0.6) node {$v^2+v^4$};
	\tikzstyle{every node}=[draw,circle,fill=black,minimum size=5pt, inner sep=0pt]
	\draw (0,0) node(a)  {};
	\draw (30:2) node [label=above:$v^{-2}$] {};
	\draw (90:2) node [label=above:$v^{-4}$] {};
	\draw (150:2) node [label=above:$v^{-2}$] {};
	\draw (-30:2) node [label=right:$1$] {};
	\draw (-90:2) node(b) [label=below:$v^2$] {};
	\draw (-150:2) node [label=left:$1$] {};
	\tikzstyle{every node}=[]
	\draw[->,thick, red] (a) -- node[above,sloped] {$\delta-\alpha_{12}^\vee$} (b);
	\end{tikzpicture}
	
\end{tabular}
\end{center}

\end{example}

\section{Charge Statistics and Swapping Functions}\label{CrystalSection}

We consider now the split reductive group $G$. Let $T\cu G$ be a maximal torus. Recall that $X$ is the set of characters of $T$. For any dominant weight $\lambda\in X_+$, let $L(\lambda)=\Sat(\IC_\lambda)$ be the corresponding irreducible module of $G$ with highest weight $\lambda$.

In \Cref{sec:WallCrossing} we have discussed
wall crossing in hyperbolic localization for the affine Grassmannian $\Gr^\vee$ of $G^\vee$, the Langlands dual group of $G$. In this section we want to define an operation on the category of representations of $G$ that corresponds to wall crossing via geometric Satake.

For $\lambda\in X_+$, let $\calB(\lambda)$ be the crystal corresponding to the representation $L(\lambda)$.
There are several equivalent different constructions of $\calB(\lambda)$. 
The crystal $\calB(\lambda)$ can be constructed algebraically using the quantum group of $G$ \cite{KasCrystal,LusCanonical} or geometrically via the  geometric Satake correspondence \cite{BGCrystals,BFGUhlenbeck}. In this case the elements of the crystal are the MV cycles and the crystal operators are induced by a geometric operation on cycles. It is shown in \cite{KamCrystal} that the algebraic and geometric constructions are equivalent.

\

For $T\in \calB(\lambda)$ we denote by $\wt(T)\in X$ the weight of $T$.
Let $\calB(\lambda)_\mu$ denote the set of elements of the crystal of weight $\mu$. Let $\calB^+(\lambda)=\{ T\in \calB(\lambda) \mid \wt(T)\in X_+\}$.
Recall the definition of charge from \cite{LSLe,LLTCrystal}.
\begin{definition}
	A \emph{charge statistic} on $\calB(\lambda)$ is a function $c:\calB^+(\lambda)\ra \bbZ_{\geq 0}$ such that for any $\mu,\lambda\in X_+$ with $\mu\leq \lambda$ we have
	\[h_{\mu,\lambda}(q^{\frac12})=K_{\lambda,\mu}(q)=\sum_{T\in \calB(\lambda)_\mu} q^{c(T)}\in \bbZ[q].\]
\end{definition}

It is convenient to extend the definition of charge to renormalized $\eta$-Kazhdan--Lusztig polynomials. 

\begin{definition}
	Let $\eta\in \affX$ be a regular cocharacter.
	A \emph{renormalized charge statistic} (or \emph{recharge} for short) for $\eta$ on $\calB(\lambda)$ is a function $r(\eta,-):\calB(\lambda) \ra \frac12\bbZ$
	such that 
	\[\htil^\eta_{\mu,\lambda}(q^{\frac12})=\sum_{T\in \calB(\lambda)_\mu} q^{r(\eta,T)}\in \bbZ[q^{\frac12},q^{-\frac12}].\]
\end{definition}

By \Cref{KLprop}, if $\eta$ is in the KL chamber and $r(\eta,-)$ is a recharge for $\eta$, then we obtain a charge statistic $c$ by setting \begin{equation}\label{charge}
c(T):=r(\eta,T)+\frac12 \ell(\wt(T)).
\end{equation}

\begin{remark}

In the literature a charge statistic is usually only defined on the dominant part of the crystal, whereas we define the recharge on the whole crystal.
However, this distinction is not substantial: if we allow the charge statistic to take values in $\frac{1}{2} \mathbb{Z}$, it can always be naturally extended from the dominant part to a statistic $c$ on the whole crystal such that
	\[h_{\mu,\lambda}^\eta(q^{\frac12})=\sum_{T\in \calB(\lambda)_\mu} q^{c(T)}\in \bbZ[q]\]
for any $\mu\in X$.

			In fact, the sheaf $\IC_\lambda$ is locally constant on $G^\vee[[t]]$-orbits, so its stalk in $L_\mu$ and $L_{w(\mu)}$ are isomorphic, for any $\mu \in X$ and $w \in W$. It follows by \eqref{KLpropEq} that 
		\begin{equation*} \htil_{w(\mu),\lambda}^\eta(v)=\grdim(\HL_{w(\mu)}^\eta(\IC_\lambda))=\grdim (i_{w(\mu)}^*\IC_\lambda[-2\ell(w(\mu))])= \htil_{\mu,\lambda}^\eta(v)v^{2(\ell(\mu)-\ell(w(\mu)))}
		\end{equation*}
		and that $h_{w(\mu),\lambda}^\eta(v)=h_{\mu,\lambda}^\eta(v)v^{\ell(\mu)-\ell(w(\mu))}$.
		Thus, we can simply extend the definition of $c$ to $\calB(\lambda)$ by setting \[c(w(T)):=c(T)+\frac12(\ell(\wt(T))-\ell(w(\wt(T))))\] for any $T\in \calB^+(\lambda)$ and $w\in W$.
\end{remark}

Assume that $\eta_1$ and $\eta_2$ are in adjacent $\lambda$-chambers as in \Cref{sec:WallCrossing}, that is, there exists a unique root $\alpha^\vee\in \affPhi^\vee_+(\lambda)$ such that $\langle \eta_1,\alpha^\vee\rangle<0$ and $\langle \eta_2,\alpha^\vee\rangle>0$. Suppose that we are given a recharge $r(\eta_1,-)$ for $\eta_1$. We want to construct a recharge for $\eta_2$. 

\begin{lemma}\label{functionpsi}
	Let $\mu\in X$ and $\nu=s_{\alpha^\vee}(\mu)$. Assume that $\mu<\nu\leq \lambda$. Then there exists an injective function $\psi_{\nu}:\calB(\lambda)_{\nu} \ra \calB(\lambda)_\mu$ such that
	\begin{equation}\label{findf}
	r(\eta_1,\psi_\nu(T))=r(\eta_1,T)-1 \qquad \text{ for every }T\in \calB(\lambda)_\nu. 
	\end{equation}
\end{lemma}
\begin{proof}
	This follows from \Cref{crossing}. In fact, there exist $p(q),p'(q)\in \bbZ_{\geq 0}[q^{\frac12},q^{-\frac12}]$ such that 
	\[\sum_{T \in \calB(\lambda)_\mu}q^{r(\eta_1,T)}=q^{-1}p(q)+p'(q)\aand\sum_{T \in \calB(\lambda)_\nu}q^{r(\eta_1,T)}=p(q).\qedhere\]
\end{proof}

\begin{definition}
	We call a collection of maps $\psi=\{\psi_\nu\}_{s_{\alpha^\vee}(\nu)<\nu\leq\lambda}$ such that each $\psi_\nu$ satisfies \cref{findf}, a \emph{swapping function} between $\eta_1$ and $\eta_2$.
\end{definition}

Given a recharge $r(\eta_1,-)$, every swapping function $\psi$ determines a recharge for $\eta_2$.
In fact, $r(\eta_2,-)$ can be obtained from $r(\eta_1,-)$ as follows. If $\mu\leq \lambda$ is such that $s_{\alpha^\vee}(\mu)\not\leq \lambda$ or $s_{\alpha^\vee}(\mu)=\mu$ then $r(\eta_2,-)=r(\eta_1,-)$ on $\calB(\lambda)_\mu$. 

If $\mu\leq \lambda$ is such that $\mu< s_{\alpha^\vee}(\mu)\leq \lambda$ and $\psi_{s_{\alpha^\vee}(\mu)}:\calB(\lambda)_{s_{\alpha^\vee}(\mu)}\ra \calB(\lambda)_\mu$ then $r(\eta_2,-)$ is obtained by swapping the values of $r(\eta_1,-)$ on $T$ and $\psi_{s_{\alpha^\vee}(\mu)}(T)$ for any $T\in \calB(\lambda)_{s_{\alpha^\vee}(\mu)}$ while leaving unchanged the values on $T$, for $T\in \calB(\lambda)_\mu \setminus\Ima(\psi_{s_\alpha^\vee(\mu)})$.

In formulas, we have
\begin{equation}\label{reta2}
r(\eta_2,T)=\begin{cases}r(\eta_1,T)-1& \text{if }s_{\alpha^\vee}(\wt(T))<\wt(T)\\
r(\eta_1,T)+1& \text{if }\wt(T)<s_{\alpha^\vee}(\wt(T))\leq \lambda\text{ and }T\in \Ima(\psi)\\
r(\eta_1,T)& \text{if }\wt(T)\leq s_{\alpha^\vee}(\wt(T))\text{ and }T\not\in \Ima(\psi).
\end{cases}
\end{equation}

\begin{thm}
	Let $\eta_1$ and $\eta_2$ be as above.
	If $r(\eta_1,-)$ is a recharge for $\eta_1$ on $\calB(\lambda)$ and $\psi$ is a swapping function then $r(\eta_2,-)$ defined by \cref{reta2} is a recharge for $\eta_2$.
\end{thm}
\begin{proof}
	Assume that $s_{\alpha^\vee}(\wt(T))\not \leq \lambda$ or that $s_{\alpha^\vee}(\wt(T))=\wt(T)$. Then we are in the third case of \eqref{reta2}. In this case $\HL^1_{\wt(T)}=\HL^2_{\wt(T)}$ by \Cref{HLadjacentgood}, therefore, $r(\eta_1,-)$ is also a recharge for $\eta_2$ on $\calB(\lambda)_{\wt(T)}$.
	
	If $s_{\alpha^\vee}(\wt(T)) \leq \lambda$, then the claim follows from \Cref{crossing} and \eqref{findf}.
\end{proof}

There are only finitely many roots in $\affPhi^\vee(\lambda)$, hence only finitely many $\lambda$-walls separating the MV chamber from the KL chamber. 
We can always find a sequence of regular cocharacters $\eta_0,\eta_1,\ldots,\eta_M$ such that $\eta_0$ is in the MV chamber, $\eta_M$ is in the KL chamber and such that $\eta_i$ and $\eta_{i+1}$ are in adjacent $\lambda$-chambers for any $i$.

There exists a unique recharge for $\eta_0$ in the MV chamber. By \Cref{KLinMV}, for any $T$ we have 
\begin{equation}\label{rMV}
r(\eta_0,T)=-\langle \wt(T),\rho^\vee\rangle.
\end{equation}

So, if we have for any $i$ a swapping function between $\eta_i$ and $\eta_{i+1}$, starting with $r(\eta_0,-)$ we can inductively construct a recharge on the KL chamber, hence a charge. 
We say that a recharge obtained through this process is a recharge \emph{obtained via swapping operations}.


\subsection{Swapping functions in  rank one}\label{rankoneSec}
We now examine in detail the case of a root system $\Phi$ of type $A_1$.
Although this is a simple setting because all weight multiplicities for irreducible representations are one, it is useful to illustrate the
procedure developed in the previous sections.
Moreover, it also serves as the base step for the inductive process used in \cite{ChargeA} to construct the charge statistic in type $A_n$ and in \cite{PTAtoms} for type $C_2$

We have $\Phi=\{ \alpha,-\alpha\}$ and $\affPhi^\vee_+=\{m\delta +\alpha^\vee \mid m\geq 0 \} \cup \{ m\delta -\alpha^\vee \mid m \geq 1\}$.
The hyperplanes that divide the MV region from the KL region are all of the form $H_{m\delta-\alpha^\vee}$, with $m\geq 1$.

Let $\lambda=N\varpi$ with $N\geq 0$. Then $\calB(\lambda)$ is isomorphic as a graph to an oriented string with $N+1$ elements and $\affPhi^\vee(\lambda)=\{ k\delta -\alpha^\vee \mid 1\leq k \leq N-1\} \cup \{ k\delta +\alpha^\vee \mid 0\leq k \leq N-1\}$.

For any $\mu\leq \lambda$,  the crystal $\calB(\lambda)$ contains exactly one element of weight $\mu$. Hence, there exists a unique swapping function $\psi$ and a unique recharge function for any cocharacter $\eta$.

Let $\eta_0$ be in the MV chamber. If $T\in \calB(N\varpi)_{A\varpi}$, with $A\in \bbZ$ then $r(\eta_0,T)=-\frac12 A$.
We choose a sequence of regular cocharacters $\eta_0,\ldots,\eta_{N-1}$ such that the only wall separating $\eta_i$ from $\eta_{i+1}$ is $H_{(N-1-i)\delta-\alpha^\vee}$.
We have 
\[s_{i\delta-\alpha^\vee}(A\varpi)=(-A-2i)\varpi\]
and $A\varpi<s_{i\delta-\alpha^\vee}(A\varpi)$ if and only if $A>-i$.
In this case the swapping function $\psi$ between $\eta_{N-i-1}$ and $\eta_{N-i}$ is given by
\[ \psi_{(-A-2i)\varpi}:\calB(N\varpi)_{(-A-2i)\varpi}\ra 
\calB(N\varpi)_{A\varpi},\qquad
\psi(T)=
e_1^{A+i}(T)\]
for any $A\leq N$ such that $A+i>0$ and $A+2i\leq N$.

We can spell out \cref{reta2} in our setting.
For $T\in \calB(N\varpi)_{A\varpi}$ we have
\begin{equation}\label{rforA1}
	r(\eta_{N-i},T)=\begin{cases}r(\eta_{N-i-1},T)-1& \text{if }A+i<0\text{ and }A+2i\leq N\\
		r(\eta_{N-i-1},T)+1& \text{if }A+i>0\text{ and }A+2i\leq N\\
		r(\eta_{N-i-1},T)& \text{if }A+2i>N\text{ or if }A+i=0.
	\end{cases}
\end{equation}

\begin{example}
	Consider the crystal $\calB(4\varpi)$. 	
	\begin{center}
		\begin{tikzpicture}
			\node at (-7,0) {};
			\tikzstyle{every node}=[draw,circle,fill=black,minimum size=5pt, inner sep=0pt]
			\draw (-4,0) node (-4) {};
			\draw (-2,0) node (-2) {};
			\draw (0,0) node (0) {};
			\draw (2,0) node (2) {};
			\draw (4,0) node (4) {};
			\tikzstyle{every node}=[]
			\path (-4) ++(0,0.4) node {$-4\varpi$};
			\path (-2) ++(0,0.4) node {$-2\varpi$};
			\path (0) ++(0,0.4) node {$0$};
			\path (2) ++(0,0.4) node {$2\varpi$};
			\path (4) ++(0,0.4) node {$4\varpi$};
			\path[->] 
			(4) edge node[above] {$f_1$} (2)
			(2) edge node[above] {$f_1$} (0)
			(0) edge node[above] {$f_1$} (-2)
			(-2) edge node[above] {$f_1$} (-4);

		\end{tikzpicture}
	\end{center}
	
	We start with the recharge $r(\eta_0,-)$ and after performing the ``swapping operation'' as dictated by \eqref{rforA1}, we obtain the recharge in the KL chamber.
	\begin{center}
		\begin{tikzpicture}
			\node at (-7,0) {};
			\tikzstyle{every node}=[draw,circle,fill=black,minimum size=5pt, inner sep=0pt]
			\draw (-4,0) node (-4) {};
			\draw (-2,0) node (-2) {};
			\draw (0,0) node (0) {};
			\draw (2,0) node (2) {};
			\draw (4,0) node (4) {};
			\tikzstyle{every node}=[]
			\path (-4) ++(0,0.4) node {$2$};
			\path (-2) ++(0,0.4) node {$1$};
			\path (0) ++(0,0.4) node {$0$};
			\path (2) ++(0,0.4) node {$-1$};
			\path (4) ++(0,0.4) node {$-2$};
			\node at (-6,0) {$r(\eta_0,-)$:};
			\begin{scope}[yshift=-1.75cm]
				\tikzstyle{every node}=[draw,circle,fill=black,minimum size=5pt, inner sep=0pt]
				\draw (-4,0) node (-4) {};
				\draw (-2,0) node (-2) {};
				\draw (0,0) node (0) {};
				\draw (2,0) node (2) {};
				\draw (4,0) node (4) {};
				\tikzstyle{every node}=[]
				\path (-4) ++(0,0.4) node {$1$};
				\path (-2) ++(0,0.4) node {$2$};
				\path (0) ++(0,0.4) node {$0$};
				\path (2) ++(0,0.4) node {$-1$};
				\path (4) ++(0,0.4) node {$-2$};
				\node at (-6,0) {$r(\eta_1,-)$:};
				\path[->,red,very thick] (-2) edge node[above] {\scriptsize$3\delta-\alpha^\vee$} (-4);
			\end{scope}
			\begin{scope}[yshift=-3.5cm]
				\tikzstyle{every node}=[draw,circle,fill=black,minimum size=5pt, inner sep=0pt]
				\draw (-4,0) node (-4) {};
				\draw (-2,0) node (-2) {};
				\draw (0,0) node (0) {};
				\draw (2,0) node (2) {};
				\draw (4,0) node (4) {};
				\tikzstyle{every node}=[]
				\path (-4) ++(0,0.4) node {$0$};
				\path (-2) ++(0,0.4) node {$2$};
				\path (0) ++(0,0.4) node {$1$};
				\path (2) ++(0,0.4) node {$-1$};
				\path (4) ++(0,0.4) node {$-2$};
				\node at (-6,0) {$r(\eta_1,-)$:};
				\path[->,red,very thick] (0) edge[out=147,in=33] node[above] {\scriptsize$2\delta-\alpha^\vee$} (-4);
			\end{scope}
			\begin{scope}[yshift=-5.25cm]
				\tikzstyle{every node}=[draw,circle,fill=black,minimum size=5pt, inner sep=0pt]
				\draw (-4,0) node (-4) {};
				\draw (-2,0) node (-2) {};
				\draw (0,0) node (0) {};
				\draw (2,0) node (2) {};
				\draw (4,0) node (4) {};
				\tikzstyle{every node}=[]
				\path (-4) ++(0,0.4) node {$-1$};
				\path (-2) ++(0,0.4) node {$1$};
				\path (0) ++(0,0.4) node {$2$};
				\path (2) ++(0,0.4) node {$0$};
				\path (4) ++(0,0.4) node {$-2$};
				\node at (-6,0) {$r(\eta_3,-)$:};
				\path[->,red,very thick] (2) edge[out=150,in=30] node[above] {\scriptsize$\delta-\alpha^\vee$} (-4);
				\path[->,red,very thick] (0) edge node[above] {\scriptsize$\delta-\alpha^\vee$} (-2);
			\end{scope}
		\end{tikzpicture}
	\end{center}
	
	Therefore, the charge statistic on $\calB(4\varpi)$ is as follows.
	\begin{center}
		\begin{tikzpicture}	
			\tikzstyle{every node}=[draw,circle,fill=black,minimum size=5pt, inner sep=0pt]
			\draw (-4,0) node (-4) {};
			\draw (-2,0) node (-2) {};
			\draw (0,0) node (0) {};
			\draw (2,0) node (2) {};
			\draw (4,0) node (4) {};
			\tikzstyle{every node}=[]
			\path (-4) ++(0,0.4) node {$\frac12$};
			\path (-2) ++(0,0.4) node {$\frac32$};
			\path (0) ++(0,0.4) node {$2$};
			\path (2) ++(0,0.4) node {$1$};
			\path (4) ++(0,0.4) node {$0$};
			\node at (-5.5,0) {$c$:};
			\node at (-7,0) {};
		\end{tikzpicture}
	\end{center}
\end{example}

Recall that $\ell(A\varpi)=A$ if $A\geq 0$ and $\ell(A\varpi)=-A-1$ if $A<0$. Thus, for an arbitrary $N\geq 0$,  we have $h_{A\varpi,N\varpi}(v)=v^{N-A}$ if $A\geq 0$ while $h_{A\varpi,N\varpi}(v)=v^{N+A+1}$ if $A<0$.
By \Cref{KLprop}, it follows that for $\eta_M$ in the KL region we have
\[ r(\eta_M,T)=\begin{cases}
	\frac12(N-2A) & \text{if }A\geq 0,\\
	\frac12(N+2A+2) & \text{if }A<0.
\end{cases}\]

We can  compute the total variation of the recharge.
\begin{lemma}\label{A1Cor}
	For any $T\in \calB(N\varpi)$ we have
	\[r(\eta_M,T)-r(\eta_0,T)=\phi_\alpha(T)-\ell(T).\]
\end{lemma}
\begin{proof}
	Observe that
	$\phi_\alpha(T)=\frac12(N+A)$.
	If $A\geq 0$ the claim reduces to $\frac12 (N-2A)+\frac12 A=\frac12 (N+A)-A$. The case $A<0$ is similar.
\end{proof}

\def\cprime{$'$}

\bibliography{/home/leo/Dropbox/Matherial/Bibliography/mybiblio}
\bibliographystyle{alpha}
\Address

\end{document}